\documentclass[12pt]{amsart}
\usepackage{times}
\DeclareMathAlphabet{\mathpzc}{OT1}{pzc}{m}{it}

\usepackage[T1]{fontenc}
\usepackage{dsfont}
\usepackage{mathrsfs}
\usepackage[colorlinks]{hyperref}
\usepackage{xcolor}
\usepackage[a4paper,asymmetric]{geometry}
\usepackage{mathscinet}
\usepackage{fullpage}
\usepackage{latexsym}
\usepackage{graphicx}
\usepackage{epstopdf}
\usepackage{amsthm}
\usepackage{amssymb}
\usepackage{amsfonts}
\usepackage{amsmath}
\usepackage{cite}
\usepackage{float}  
\usepackage{subfigure} 
\usepackage{caption} 
\usepackage[ruled,linesnumbered]{algorithm2e}
\usepackage[square]{natbib}\setcitestyle{numbers,list}
\newtheorem{theorem}{Theorem}[section]
\newtheorem{thm}[theorem]{Theorem}

\newtheorem{lemma}[theorem]{Lemma}
\newtheorem{lem}[theorem]{Lemma}
\newtheorem{proposition}[theorem]{Proposition}

\newtheorem{corollary}[theorem]{Corollary}
\newtheorem{assumption}[theorem]{Assumption}
\theoremstyle{definition}

\newtheorem{defn}[theorem]{Definition}

\theoremstyle{remark}
\newtheorem{remark}[theorem]{Remark}
\newtheorem{rem}[theorem]{Remark}
\numberwithin{equation}{section}

 \DeclareMathAlphabet{\mathpzc}{OT1}{pzc}{m}{it}

  \newcommand{\dif}{\mathrm{d}}

 \newcommand{\Y}{\mathcal{Y}}

 \newcommand{\W}{\mathcal{W}}           
 \newcommand{\cov}{\mathrm{cov}}       
 \newcommand{\E}{\mathbb{E}}            
 
 \newcommand{\e}{\varepsilon}

 \newcommand{\Ll}{\langle}
 \newcommand{\Rr}{\rangle}

 \newcommand{\R}{\mathbb{R}}
 
 \newcommand{\PP}{\mathbb{P}}
 \newcommand{\mcl}{\mathcal}
 
 \newcommand{\Be}{\begin{equation}}
 \newcommand{\Ee}{\end{equation}}
  \newcommand{\Bes}{\begin{equation*}}
 \newcommand{\Ees}{\end{equation*}}
  \newcommand{\Bey}{\begin{eqnarray}}
 \newcommand{\Eey}{\end{eqnarray}}
 \newcommand{\Beys}{\begin{eqnarray*}}
 \newcommand{\Eeys}{\end{eqnarray*}}
 \newcommand{\BT}{\begin{thm}}
 \newcommand{\ET}{\end{thm}}
 \newcommand{\Bp}{\begin{proof}}
 \newcommand{\Ep}{\end{proof}}
 \newcommand{\BL}{\begin{lem}}
 \newcommand{\EL}{\end{lem}}
 \newcommand{\BP}{\begin{proposition}}
 \newcommand{\EP}{\end{proposition}}
 \newcommand{\BC}{\begin{corollary}}
 \newcommand{\EC}{\end{corollary}}
 \newcommand{\BR}{\begin{rem}}
 \newcommand{\ER}{\end{rem}}
 \newcommand{\BD}{\begin{defn}}
 \newcommand{\ED}{\end{defn}}
 \newcommand{\BI}{\begin{itemize}}
 \newcommand{\EI}{\end{itemize}}

\usepackage{marginnote}
\begin{document}

\title[]
{Self-normalized Cram\'er-type Moderate Deviation of Stochastic Gradient Langevin Dynamics}

\author[H. Dai]{Hongsheng Dai}
\address[H. Dai]{School of Mathematics, Statistics and Physics, Newcastle University, UK}
\email{hongsheng.dai@newcastle.ac.uk}

\author[X. Fan]{Xiequan Fan}
\address[X. Fan]{School of Mathematics and Statistics, Northeastern University at Qinhuangdao, China}
\email{fanxiequan@hotmail.com}

\author[J. Lu]{Jianya Lu$^*$}
\address[J. Lu]{School of Mathematics, Statistics and Actuarial Science, University of Essex, UK}
\email{jianya.lu@essex.ac.uk}

\maketitle

\begin{abstract}

In this paper, we study the self-normalized Cram\'er-type moderate deviation of the empirical measure of the stochastic gradient Langevin dynamics (SGLD). Consequently, we also derive the Berry-Esseen bound for SGLD. Our approach is by constructing a stochastic differential equation (SDE) to approximate the SGLD and then applying Stein’s method, as developed in \cite{FSX19, Lu2022central},  to decompose the empirical measure into a martingale difference series sum and a negligible remainder term.

	
\end{abstract}

\section{Introduction}\label{sec:intro}

For a nonconvex stochastic loss function $\psi(\omega,\zeta): \R^d\times\R^r\to\R$, where $\zeta\in\R^r$ is a random variable with probability distribution $\nu$, we consider the following optimization problem
\begin{eqnarray*}
	\omega^*=\text{argmin}_{\omega\in\R^d}P(\omega),\quad P(\omega)=\E_{\zeta\sim\nu}\psi(\omega,\zeta).
\end{eqnarray*}
To find the minimizer $\omega^*$, \cite{welling2011bayesian} proposed the stochastic gradient Langevin dynamic (SGLD) algorithm, which has been widely applied to optimization problems. The iteration of the SGLD is given as follows,
\begin{eqnarray}\label{e:sgld}
	\omega_k=\omega_{k-1}-\eta\nabla\psi(\omega_{k-1},\zeta_k)+\sqrt{\eta\delta}\xi_{k},
\end{eqnarray}	
where $\eta>0$ is the step size, $\delta>0$ is the inverse temperature parameter, $(\xi_k)_{k\ge1}$ are sequence of independent and identically distributed (i.i.d.) standard $d$-dimensional normal random vectors and $(\zeta_k)_{k\ge1}$ are i.i.d.\ samples from $\nu$.

As the number of iterations $k$ tends to infinity, \cite{raginsky2017non} showed that \eqref{e:sgld} can find the approximate global minimizer. See \cite{xu2018global, lamperski2021projected, chen2020stationary} for more details on the convergence of the SGLD. Unlike the stochastic gradient descent (SGD) algorithm, which may converge to local minima in non-convex optimization problems, the SGLD algorithm benefits from the inclusion of Gaussian noise in its iterations. This added noise allows SGLD to more effectively explore the parameter space, making it well-suited for solving non-convex problems
\cite{chen2020stationary, zhang2017hitting}.


For the iteration \eqref{e:sgld}, it is natural to consider it as a discretization to a continuous dynamic for a given step size $\eta$. We consider the following stochastic differential equation (SDE) to approximate the SGLD algorithm,
\begin{eqnarray}\label{e:sde}
	\dif X_t=-\nabla P(X_t)\dif t+Q_{\eta,\delta}(X_t)\dif B_t,
\end{eqnarray}
where $B_t$ is a $d-$dimensional standard Brownian motion and the diffusion matrix $Q_{\eta,\delta}(\cdot)\in\R^{d\times d}$ will be defined later. Significant work has been done in \cite{feng2019uniform, hu2019diffusion} on comparing stochastic algorithms to their  corresponding SDE approximations, and on establishing the diffusion approximation bound of 
\begin{eqnarray}\label{e:convergence}
	\sup_{h\in\mathcal H}|\E h(\omega_k)-\E h(X_{k\eta})|
\end{eqnarray}
for a family $\mathcal H$ of test functions $h$. Different choices of  $\mathcal{H}$ correspond to different distance metrics, such as the Wasserstein$-1$ distance for the Lipschitz function $h$ and the total variation distance for bounded $h$. The diffusion approximation provides valuable insights into algorithms by viewing them as continuous dynamics. Acting as a bridge, it enables the application of continuous dynamic analysis methods to study the properties of stochastic algorithms. See \cite{li2017stochastic, li2019stochastic, chen2022approximation}  for more details.

Under suitable conditions on $\psi$, \eqref{e:sgld} and \eqref{e:sde} are exponential ergodic with invariant measures $\pi_\eta$ and $\pi$ respectively. For the SGLD and its invariant measure $\pi_\eta$, we construct an empirical measure
$\Pi_\eta$ as a statistic of $\pi_\eta$, where
\begin{eqnarray*}
	\Pi_\eta(\cdot)=\frac{1}{m}\sum_{k=0}^{m-1}\delta_{\omega_k}(\cdot).
\end{eqnarray*}
Here $\delta_{\omega_k}(\cdot)$ is the Dirac measure of $\omega_k$. Since \eqref{e:convergence} converges to zero as $k\to\infty$ and $\eta\to0$, given test function $h:\R^d \rightarrow \R$, it is natural to consider the asymptotic property of $\Pi_\eta(h)$, i.e.,  $\int h\dif \Pi_\eta$.

The study of self-normalized Cram\'er type moderate deviation (SNCMD) explores the deviation properties of random variables and has developed over the last decades, see \cite{shao1999cramer, JSW03} for the results of independent random variables. For dependent random variables, \cite{CSWX16} studied the moderate deviation under mixing conditions, \cite{Fan2019self} focused on the self-normalized Cram\'er type moderate deviation for martingales, and \cite{Fan2020} analysed stationary sequences. We refer the reader to \cite{jiang2022deviation, shao2016cramer, zhang2023cramer} for further reference.  However, for the iterative output of a stochastic algorithm, which is a sequence of dependent random variables, there has been limited analysis of SNCMD for it. See  \cite{Teh2016, gao2023moderate} for the fluctuation analysis of stochastic algorithms. 

 In this paper, we analyse the self-normalized Cram\'er type moderate deviation of $\Pi_\eta(h)$ with Lipschitz test function $h$. Specifically, given a normalized term $\mathcal{Y}_\eta$, we compare the tail probability of $\Pi_\eta(h)/\sqrt{\mathcal{Y}_\eta}$
 after scaling and centralization, i.e., $\sqrt{m\eta}(\Pi_\eta(h)-\pi(h))/\sqrt{\delta\mathcal{Y}_\eta}$,
 with the tail probability of a standard normal distribution $N(0,I_d)$. Consequently, we also derive the Berry-Esseen bound. 
 
Using the diffusion approximation and Stein's method, we have, for the first time, investigated the self-normalized Cram\'er type moderate deviation of the SGLD algorithm, which provides a novel approach and perspective on the analysis of the asymptotic properties of the SGLD algorithm. Although \cite{Lu2022central, fan2024normalized} carried out a similar analysis, i.e. examining the SNCMD of the Langevin dynamics, their results are based on the relatively restrictive assumption that both the gradient and test functions are bounded.
In contrast, our results extend this assumption by replacing the boundedness assumption with a Lipschitz condition. To relax this condition, we employ a truncation technique in the proof. Additionally, our work provides a new application of Stein’s method within the realm of machine learning.

The approach to proving the main results relies on Stein’s method and a standard decomposition of $\Pi_\eta(h)$, with similar ideas found in \cite{Lu2022central, Teh2016}. The strategy of the proof begins with a diffusion approximation for the stochastic algorithm, constructing a corresponding SDE. Under some mild conditions, the SDE has an ergodic measure $\pi$, and its associated Stein's equation has a solution with good regularity properties. Using Stein's equation, we decompose $\Pi_\eta(h)$ into a martingale difference series sum $\mathcal{H}_\eta$ and a remainder $\mathcal{R}_\eta$. For $\mathcal{H}_\eta$, we apply the martingale SNCMD theorem in \cite{fan2024normalized} to compare it with the standard normal distribution. Additionally, we show that the remainder $\mathcal R_\eta$ is exponentially negligible using concentration inequalities.

The paper is organized as follows. Diffusion approximation and our main results are stated in Section \ref{sec:2}. In Section \ref{sec:3}, we provide some preliminary lemmas. 
In Section \ref{sec:5}, we give the proof of SNCMD and Berry-Esseen bound. The details of the proof of preliminary lemmas are deferred to Sections \ref{appen:regularity} and \ref{appen:3}. 

We finish this section by introducing some notations which will be frequently used in the sequel. The Euclidean norm and the inner product of vectors $x, y \in \R^d$ are denoted by $|x|$ and $\langle x, y \rangle$, respectively. For higher-rank tensors, we denote their norm by $\|\cdot\|$. For any two matrices $A,B\in \R^{d\times d}$, the Hilbert-Schmidt norm is denoted by
$\|A\|_{\text{HS}}=\sqrt{\sum_{i,j=1}^d A_{ij}^2}=\sqrt{\textrm{Tr}(A^\top A)}$ and their inner product by 
$\Ll A,B\Rr_{\text{HS}}:=\sum_{i,j=1}^d A_{ij}B_{ij}$, where $\top$ is the transpose operator. The symbols $C$ and $c$ represent positive constants whose values may vary from line to line. Let $\text{Lip}_1(\R^d)$ be the family of Lipschitz function defined on $\R^d$ with Lipschitz constant $1$. We denote the conditional expectation $\E[\cdot|\omega_k]$ and conditional probability $\mathbb{P}(\cdot|\omega_k)$ by $\E_k[\cdot]$ and  $\mathbb{P}_k(\cdot)$, respectively. Finally, $\Phi(x)$ represents the cumulative distribution function for standard normal random variables.

\section{Diffusion approximation and Main results}\label{sec:2}


We first construct the diffusion approximation. Rewriting \eqref{e:sgld}, we have
\begin{eqnarray}\label{e:sgld1}
	\omega_k&=&\omega_{k-1}-\eta\nabla P(\omega_{k-1})+\eta\nabla P(\omega_{k-1})-\eta\nabla\psi(\omega_{k-1},\zeta_k)+\sqrt{\eta\delta}\xi_{k}\nonumber\\
	&:=&\omega_{k-1}-\eta\nabla P(\omega_{k-1})+\sqrt\eta V_{\eta,\delta}(\omega_{k-1},\zeta_k,\xi_{k}),
\end{eqnarray}
where 
\begin{eqnarray*}
	V_{\eta,\delta}(\omega_{k-1},\zeta_k,\xi_{k})&=& \sqrt\eta\nabla P(\omega_{k-1})-\sqrt\eta\nabla\psi(\omega_{k-1},\zeta_k)+\sqrt{\delta}\xi_{k}.
\end{eqnarray*}
As  $\E\psi(\cdot,\zeta)=P(\cdot)$, a straightforward calculation implies
\begin{eqnarray*}
	\E_{k-1}[V_{\eta,\delta}(\omega_{k-1},\zeta_k,\xi_{k})]&=&0
\end{eqnarray*}
and
\begin{eqnarray*}
	\cov[V_{\eta,\delta}(\omega_{k-1},\zeta_k,\xi_{k})|\omega_{k-1}]=\eta\Sigma(\omega_{k-1})+\delta I_d,
\end{eqnarray*}
where 
\begin{eqnarray*}
	\Sigma(x)=\E[\nabla\psi(x,\zeta)\nabla\psi(x,\zeta)^\top]-\nabla P(x)\nabla P(x)^\top
\end{eqnarray*}
and $I_d$ is the $d-$dimensional identity matrix. Thus it is natural to consider the following SDE to approximate \eqref{e:sgld}, 
\begin{eqnarray}\label{e:sde1}
	\dif X_t=-\nabla P(X_t)\dif t+Q_{\eta,\delta}(X_t)\dif B_t,
\end{eqnarray}
where 
\begin{eqnarray*}
	Q_{\eta,\delta}(x)=\big(\eta\Sigma(x)+\delta I_d\big)^{\frac12}
\end{eqnarray*}
is a positive definite matrix and $B_t$ is a $d-$dimensional standard Brownian motion. For the cost function and random variable $\zeta$, we introduce the following conditions.

\begin{assumption}\label{assu1}
	There exist constants $L, K_1 >0$ and $K_2 \ge 0$ such that for every $x,y\in \R^d$, $z\in\R^r$, 
	\begin{eqnarray}\label{e:lip}
		|\nabla\psi(x,z)-\nabla\psi(y,z)|\le L|x-y|,
	\end{eqnarray}
	\begin{eqnarray}\label{e:dissi}
		\Ll x-y,-\nabla \psi(x,z)+\nabla \psi(y,z)\Rr\le -K_1|x-y|^2+K_2.
	\end{eqnarray}
\end{assumption}

\begin{assumption}\label{assu2}
	The random variable $\nabla\psi(x,\zeta)$ is sub-Gaussian for any $x\in\R^d$, that is, there exist positive constants $K_\zeta$ and $C$ such that 
	\begin{eqnarray*}
		\E\exp\{K_\zeta|\nabla\psi(x,\zeta)|^2\} \le C.
	\end{eqnarray*} 
\end{assumption}
\begin{remark}
For the ease of the proof, we assume that $K_\zeta=1$.
\end{remark}

For the SGLD algorithm \eqref{e:sgld} and its corresponding SDE \eqref{e:sde1}, Lemma \ref{lem:ergodic} below implies that they are exponential ergodic with invariant measures $\pi_\eta$ and $\pi$ respectively. Then we have the following Wasserstein-1 distance bound between $\pi$ and $\pi_\eta$.  
\begin{theorem}\label{lem: pipieta}
Suppose Assumptions \ref{assu1} and \ref{assu2} hold, one has
\begin{eqnarray}
W_1(\pi,\pi_\eta)\le C\eta^{1/2},
\end{eqnarray}
where
\begin{eqnarray*}
	W_1(\pi,\pi_\eta)=\sup_{h\in\mathrm{Lip_1}}|\pi(h)-\pi_\eta(h)|
\end{eqnarray*}
is the Wasserstein-$1$ distance.
\end{theorem}

Let $f$ be the solution to the following Stein's equation:
\begin{align}\label{e:stein}
	h-\pi(h)=\mathcal{L}f,
\end{align}
where $\mathcal L$ is the generator of \eqref{e:sde1} given by 
\begin{equation}\label{e:generator}
	\mathcal L g(x)=\Ll -\nabla P(x), \nabla g(x) \Rr+\frac1 2\Ll Q_{\eta,\delta}(x),\nabla^2g(x) \Rr_{\text{HS}}\,, 
	\quad g\in\mathcal D(\mathcal L).
\end{equation}
Denote
\begin{align*}
	\Y_\eta=\frac{1}{m}\sum_{k=0}^{m-1} |\nabla f(\omega_k)|^2,\quad 	\W_\eta= \frac{\sqrt{m\eta}(\Pi_\eta(h)- \pi(h))}{\sqrt{\delta	\Y_\eta}}.
\end{align*}
We have the following SNCMD of the SGLD algorithm.

\begin{theorem}\label{thm-self-normal}
Suppose Assumptions \ref{assu1} and \ref{assu2} hold. Let $\omega_0\sim\pi_\eta$ and $h \in \mathrm{Lip}_1(\R^d,\R)$, set $m=o(\eta^{-2})$ with $m\eta\to\infty$. Then we have:

(i) as $m\le\eta^{-\frac{13}8}\delta^{-\frac98}$,
\begin{eqnarray*}
	\Big| \frac{\PP(\mathcal W_\eta> x)}{1-\Phi(x)}
	-1 \Big| & \le & 
	C\big(x^3m^{-\frac14}+(1+x)m^{-\frac14}\ln m+x^6\eta^{\frac12}\delta^{-\frac12}\big)
\end{eqnarray*}
uniformly for $0\le x=o(\eta^{{-\frac{1}{12}}}\delta^{\frac{1}{12}})$ as $\eta$ tends to zero and $m$ tends to infinity.  

(ii) as $m>\eta^{-\frac{13}8}\delta^{-\frac98}$,
\begin{eqnarray*}
	\Big| \frac{\PP(\mathcal W_\eta> x)}{1-\Phi(x)}
	-1 \Big| & \le & 
	C\big(x^3(m\eta\delta)^{-\frac12}+\sqrt m\eta\delta x+ m^{-\frac14}\ln m\big)
\end{eqnarray*}
uniformly for $0\le x=o( (m\eta\delta)^{\frac16}\wedge (\sqrt m\eta\delta)^{-1})$ as $\eta$ tends to zero and $m$ tends to infinity. 

Moreover, the same results hold when $\mathcal{W}_\eta$ is replaced by $-\mathcal{W}_\eta$.
\end{theorem}

Based on Theorem \ref{thm-self-normal}, we also derive the Berry-Esseen bound for the SGLD algorithm. 

\begin{theorem}\label{coro:berry}
Under the assumption of Theorem \ref{thm-self-normal}, we have
\begin{eqnarray*}
	\sup_{x\in\R}|\PP(\mathcal{W}_\eta\le x)-\Phi(x)|\le Cm^{-\frac14}\ln m.
\end{eqnarray*}	
Further assume $m=C\eta^{-2}/|\ln\eta|$, we have
\begin{eqnarray*}
	\sup_{x\in\R}|\PP(\mathcal{W}_\eta\le x)-\Phi(x)|\le C\eta^{\frac12}|\ln \eta|^{5/4}.
\end{eqnarray*}	
\end{theorem}
By choosing a specific connection between the step size $\eta$ and the number of iterations $m$, we can achieve the bound $\sqrt \eta|\ln\eta|^{5/4}$ which is close to the convergence rate in Theorem \ref{lem: pipieta}, up to a logarithmic correction. This result also has the same order as that in \cite{fan2024normalized} up to a logarithmic correction, though their result is under stronger conditions.

\begin{remark}
The assumption $\omega_0\sim\pi_\eta$ in Theorems \ref{thm-self-normal} and \ref{coro:berry} is not essential. Due to the exponential ergodicity of the SGLD algorithm, one can extend it to the case in which $\omega_0$ is subgaussian distributed. The advantage of taking $\omega_0\sim\pi_\eta$ is that in the calculation, the terms describing the difference between the distribution of $\omega_k$ and $\pi_\eta$ will vanish, while in the general case, one has to use exponential ergodicity of $\omega_k$ to bound the difference. Since $\omega_k$ converges to $\pi_\eta$ exponentially fast, the difference will not put an essential difficulty. For the ease of calculation, we only considered the case of $\omega_0\sim\pi_\eta$.
\end{remark}

\section{Auxiliary lemmas for the proof}\label{sec:3}

The strategy for proving our main result is to decompose $\sqrt{m\eta/\delta}(\Pi_\eta(h)-\pi(h))$ into a martingale term and a remainder term as in \eqref{e:Decompose}, and show that the remainder is negligible and that the martingale satisfies SNCMD. In this section, we give the decomposition and some auxiliary lemmas needed for the proof.

\begin{lemma}\label{lem:1} 
	Under Assumption \ref{assu1}, we have $\nabla P$ and $Q_{\eta,\delta}$ are Lipschitz and satisfies the dissipative condition, that is,  for any $x,y\in \R^d$,
		\begin{eqnarray}\label{e:plip}
		|\nabla P(x)-\nabla P(y)|\le L|x-y|,
	\end{eqnarray}
	\begin{eqnarray}\label{e:pdissi}
		\Ll x-y,-\nabla P(x)+\nabla P(y)\Rr\le -K_1|x-y|^2+K_2,
	\end{eqnarray}
	\begin{eqnarray}\label{e:qlip}
	\|Q_{\eta,\delta}(x)-Q_{\eta,\delta}(y)\|\le C\sqrt \eta |x-y|.
\end{eqnarray}
We also have that $\nabla \psi$ and $\nabla P$ have linear growth, that is, 
		\begin{eqnarray}\label{e:psibound1}
		|\nabla\psi(x,\zeta)|\le L|x|+|\nabla\psi(0,\zeta)|,
	\end{eqnarray} 
	\begin{eqnarray}\label{e:pbound1}
		|\nabla P(x)|\le L|x|+|\nabla P(0)|.
	\end{eqnarray}
\end{lemma}
\begin{proof}
	The proof will be given in Appendix \ref{appen:lem1}.	
\end{proof}

\begin{lem}  \label{lem:ergodic}
	Under Assumption \ref{assu1}, $(\omega_k)_{k\ge0}$ and the SDE (\ref{e:sde1}) are both exponential ergodic with invariant measures $\pi_\eta$ and $\pi$ respectively.
\end{lem}
\begin{proof}
	The proof will be given in Appendix \ref{appen:ergodic}.	
\end{proof}

\begin{lemma}\label{lem:stein}
	Let $h\in \mathrm{Lip}_1(\R^d,\R)$. A solution to Stein's equation  
	\begin{eqnarray*}
		h-\pi(h)=\mathcal L f
	\end{eqnarray*}
	is given by
	\begin{eqnarray}\label{e:steinsol}
		f(x)=-\int_0^\infty \E[h(X_t(x))-\pi(h)]\dif t,
	\end{eqnarray}
	where $X_t(x)$ is the solution of equation \eqref{e:sde1} with initial value $x$. Moreover, there exists a positive constant $C$ such that
	\begin{eqnarray}
		|f(x)|&\le& C(1+|x|^2),\label{e:regularity1}\\
		|\nabla f(x)|&\le& C(1+|x|^3),\label{e:regularity2}\\
		\|\nabla ^2 f(x)\|&\le& C(1+|x|^4),\label{e:regularity3}\\
		\sup_{y:|y-x|\le1}\frac{\|\nabla ^2 f(x)-\nabla ^2 f(y)\|}{|x-y|}&\le& C(1+|x|^{5}).\label{e:regularity4}
	\end{eqnarray}
\end{lemma}
\begin{proof}
	The proof will be given in Section \ref{appen:regularity}.	
\end{proof}

Now we introduce the decomposition. By Stein's equation (\ref{e:stein}),
\begin{align*}
	\Pi_\eta(h)-\pi(h)&=\frac{1}{m}\sum_{k=0}^{m-1}\left(h(\omega_k )-\pi(h)\right)\\
	&=\frac1{m\eta}\sum_{k=0}^{m-1}\left[\mathcal{L}f(\omega_k )\eta-\left(f(\omega_{k+1} )-f(\omega_k )\right)\right]+\frac1{m\eta}\sum_{k=0}^{m-1}\left(f(\omega_{k+1} )-f(\omega_k )\right)\\
	&=\frac1{m\eta}[f(\omega_m )-f(\omega_0 )]+\frac1{m\eta}\sum_{k=0}^{m-1}\left[\mathcal{L}f(\omega_k )\eta-(f(\omega_{k+1} )-f(\omega_k ))\right].
\end{align*}
Equations (\ref{e:sgld}), (\ref{e:generator}) and the Taylor expansion yield that
\begin{align*}
	&\mathcal{L}f(\omega_k )\eta-(f(\omega_{k+1} )-f(\omega_k ))\\
	=& \frac\eta2\Ll \nabla^2f(\omega_k ), \eta\Sigma(\omega_k)+\delta I_d\Rr_\mathrm{HS}
	-\Ll \nabla f(\omega_k),\eta\nabla P(\omega_k)-\eta\nabla\psi(\omega_k,\zeta_{k+1})\Rr\\
	&-\sqrt{\eta\delta}\Ll \nabla f(\omega_k ),\xi_{k+1}\Rr
	-\int_0^1\int_0^1s\Ll\nabla^2 f(\omega_k+ss'\Delta\omega_k), \Delta \omega_k\Delta \omega_k^\top\Rr_\mathrm{HS}\dif s'\dif s,
\end{align*}
where $\Delta\omega_k=-\eta \nabla \psi(\omega_{k},\zeta_{k+1})+\sqrt{\eta\gamma} \xi_{k+1}$. Thus we have the decomposition as follows,
\begin{equation}  \label{e:Decompose}
	\frac{\sqrt{m\eta}}{\sqrt\delta}(\Pi_\eta(h)-\pi(h))= \mathcal{H}_\eta+\mathcal{R}_\eta,
\end{equation}
where, as we shall see below, $\mathcal{H}_\eta$ is a martingale and $\mathcal{R}_\eta$ is a remainder, given by
\begin{align*}
	\mathcal{H}_\eta = -\frac{1}{\sqrt{m}}\sum_{k=0}^{m-1}\Ll\nabla f(\omega_k ),\xi_{k+1}\Rr,\quad \mathcal{R}_\eta=-\sum_{i=1}^4\mathcal{R}_{\eta,i},
\end{align*}
with
\begin{align*}
	\mathcal{R}_{\eta,1}=&\frac1{\sqrt{m\eta\delta}}(f(\omega_0)-f(\omega_m )),\\
	\mathcal{R}_{\eta,2}=&\frac{\sqrt\eta}{\sqrt m\delta}\sum_{k=0}^{m-1}\Ll \nabla f(\omega_k),\nabla P(\omega_k)-\nabla\psi(\omega_k,\zeta_{k+1})\Rr,\\
	\mathcal{R}_{\eta,3}=&\frac{1}{\sqrt{m\eta\delta}}\sum_{k=0}^{m-1}
	\int_0^1\int_0^1s \Ll \nabla^2 f(\omega_k+rr'\Delta\omega_k)-\nabla^2f(\omega_k ), \Delta \omega_k\Delta \omega_k^\top\Rr_\mathrm{HS}\dif s'\dif s,\\
	\mathcal{R}_{\eta,4}=& \frac{1}{2\sqrt{m\eta\delta}}\sum_{k=0}^{m-1}\Ll \nabla^2f(\omega_k ), \eta^2\Sigma(\omega_k)+\eta\delta I_d-\Delta\omega_k\Delta\omega_k^\top\Rr_\mathrm{HS}
	\big\}.
\end{align*}

The estimation of $\mathcal H_\eta$ and $\mathcal R_\eta$ depends on the following two lemmas.
\begin{lemma}\label{lem:R}
Suppose that Assumptions \ref{assu1} and \ref{assu2} hold. Let 
$h\in \mathrm{Lip}_1(\R^d,\R)$ and $f:\R^d\to\R$ be the solution
of \eqref{e:stein}. The following inequality holds, 
\begin{eqnarray*}
\PP(|\mathcal{R}_\eta|>y)\le C\Big( e^{-cy\eta^{\frac12}\delta^\frac12m^{\frac12} }
+e^{-cy^\frac25\delta^\frac15 \eta^{-\frac1{5}}}
+e^{-cy^{\frac29}\eta^{-\frac29}\delta^{-\frac29} }
+e^{-cy^{\frac27}\delta^{\frac17}\eta^{-\frac37}}\Big),
\end{eqnarray*}
for any $y$ satisfying $c_{m,\eta}\le y\le C\eta^{-\frac72}\delta^{-\frac72}$, where $c_{m,\eta}=c(\eta^{\frac12}\delta^{-\frac12}\vee m^{\frac12}\eta\delta)$.
\end{lemma}	
\begin{proof}
The proof will be given in Section \ref{appen:3}.	
\end{proof}

\begin{lemma}{\cite[Lemma 3.5]{fan2024normalized}}\label{lem:fan}
Let $\left(\beta_i, \mathcal{F}_i\right)_{i=1, \ldots, m}$ be a finite sequence of martingale differences. Assume that the following conditions hold:
\begin{itemize}
	\item[(A1)] There exists a number $\epsilon_m \in(0, \frac{1}{2}]$ such that
	$$
	\left|\E[\beta_i^k | \mathcal{F}_{i-1}]\right| \leq \frac{1}{2} k!\epsilon_m^{k-2} \E[\beta_i^2|\mathcal{F}_{i-1}], \quad \text { for all } k \geq 2 \text { and } 1 \leq i \leq m ;
	$$
	\item[(A2)] There exist a number $\delta_m \in(0, \frac{1}{2}]$ and a positive constant $C$ such that for all $x>0$,
	$$
	\PP\big(\big|\sum_{i=1}^m\E[\beta_i^2|\mathcal{F}_{i-1}]-\E[\beta_i^2]\big| \geq x\big) \leq C \exp \left\{-x^2 \delta_m^{-2}\right\}.
	$$
\end{itemize}
Then the following inequality holds for all $0 \leq x=o\left(\min \left\{\epsilon_m^{-1}, \delta_m^{-1}\right\}\right)$,
\begin{eqnarray*}
&&\Big|\ln \frac{\PP\big(\sum_{i=1}^m\beta_i / \sqrt{\sum_{i=1}^m\E[\beta_i^2|\mathcal{F}_{i-1}]}\ge x\big)}{1-\Phi(x)}\Big| \\
&\leq& C\left(x^3\left(\epsilon_m+\delta_m\right)+(1+x)\left(\delta_m\left|\ln \delta_m\right|+\epsilon_m\left|\ln \epsilon_m\right|\right)\right) .	
\end{eqnarray*}
\end{lemma}

\section{Proof of Main result}\label{sec:5}

In this section, we present the proof of our main result. The proof of Theorem \ref{lem: pipieta} is based on the Stein method as developed in \cite[Theorem 2.5]{FSX19}. For Theorems \ref{thm-self-normal} and \ref{coro:berry}, we analyze the normalized Cramér-type moderate deviations for martingales and demonstrate that the remainder term $\mathcal{R}_\eta$ is negligible. See \cite{fan2022cram1} and \cite{fan2024normalized}  for more details.

\begin{proof}[Proof of Theorem \ref{lem: pipieta}]
Let $(\omega_k)_{k\ge0}$ be the Markov Chain with initial value $\omega_0\sim\pi_\eta$. The Taylor expansion implies that
\begin{eqnarray}\label{e:lem3.1.1}
0&=&\E[f(\omega_1)-f(\omega_0)]\nonumber\\
&=&\E\big[\Ll \nabla f(\omega_0), \Delta\omega_0 \Rr
+\int_0^1\int_0^1 s\Ll \nabla^2f\big(\omega_0+rr'\Delta\omega_0\big),\Delta\omega_0\Delta\omega_0^\top    \Rr_{\text {HS}}\dif s\dif s'\big]\nonumber\\
&=&\E[\Ll \nabla f(\omega_0), \Delta\omega_0 \Rr]+\frac12\E[\Ll \nabla^2f (\omega_0),\Delta\omega_0\Delta\omega_0^\top     \Rr_{\text {HS}}]\nonumber\\
&&+\E\big[\int_0^1\int_0^1 s\Ll \nabla^2f\big(\omega_0+ss'\Delta\omega_0)\big)-\nabla^2f(\omega_0),\Delta\omega_0\Delta\omega_0^\top    \Rr_{\text {HS}}\dif s\dif s'\big],
\end{eqnarray}
where $\Delta \omega_0=\omega_1-\omega_0$. Following \eqref{e:sgld} and \eqref{e:sgld1}, one can get
\begin{eqnarray*}
	\E[\Ll \nabla f(\omega_0), \Delta \omega_0 \Rr]&=&\E[\Ll \nabla f(\omega_0), \E_0[\Delta \omega_0] \Rr]\\
	&=&\E[\Ll \nabla f(\omega_0),-\eta\nabla P(\omega_{0})\Rr],
\end{eqnarray*}
and
\begin{eqnarray*}
	\E[\Ll \nabla^2f (\omega_0),\Delta \omega_0\Delta \omega_0^\top     \Rr_{\text {HS}}]
	&=&\E[\Ll \nabla^2f (\omega_0),\E_0[\Delta \omega_0\Delta \omega_0^\top]     \Rr_{\text {HS}}]\\
	&=&\E[\Ll \nabla^2f (\omega_0), \eta^2\nabla P(\omega_{0})\nabla P(\omega_{0})^\top+\eta^2\Sigma(\omega_0)+\delta\eta I_d    \Rr_{\text {HS}}].
\end{eqnarray*}
Recall the generator of $(X_t)_{t\ge0}$, 
\begin{eqnarray*}
\mathcal L f(\omega_0)=\Ll -\nabla P(\omega_0), \nabla f(\omega_0) \Rr+\frac1 2\Ll \eta\Sigma(\omega_0)+\delta I_d,\nabla^2f(\omega_0) \Rr_{\text{HS}}.
\end{eqnarray*}
Combining equalities above with \eqref{e:stein}, for any Lipschitz test function $h$, we obtain
\begin{eqnarray}\label{e:lem3.1.2}
\E[h(\omega_0)-\mu(h)]&=&\E[\mathcal L f(\omega_0)]\nonumber\\
&=& -\frac1\eta \E\big[\int_0^1\int_0^1 s\Ll \nabla^2f\big(\omega_0+ss'\Delta\omega_0)\big)-\nabla^2f(\omega_0),\Delta\omega_0\Delta\omega_0^\top     \Rr_{\text {HS}}\dif s\dif s'\big]\\
&&-\frac12\E[\Ll  \nabla^2f(\omega_0),\eta\nabla P(\omega_0)\nabla P(\omega_0)^\top
 \Rr_{\text{HS}}].\nonumber
\end{eqnarray}

For the integration term of \eqref{e:lem3.1.2}, one has
\begin{eqnarray*}
&&\E\big[\int_0^1\int_0^1 s\Ll \nabla^2f\big(\omega_0+ss'\Delta\omega_0)\big)-\nabla^2f(\omega_0),\Delta\omega_0\Delta\omega_0^\top    \Rr_{\text {HS}}\dif s\dif s'\big]\\
&=&\E\big[\int_0^1\int_0^1 s\Ll \nabla^2f\big(\omega_0+ss'\Delta\omega_0)\big)-\nabla^2f(\omega_0),\Delta\omega_0\Delta\omega_0^\top     \Rr_{\text {HS}}\dif s\dif s'1_{\{|\Delta\omega_0|\le1\}}\big]\\
&&+\E\big[\int_0^1\int_0^1 s\Ll \nabla^2f\big(\omega_0+ss'\Delta\omega_0)\big)-\nabla^2f(\omega_0),\Delta\omega_0\Delta\omega_0^\top     \Rr_{\text {HS}}\dif s\dif s'1_{\{|\Delta\omega_0|>1\}}\big].
\end{eqnarray*}
For the first term above, \eqref{e:regularity4} implies that
\begin{eqnarray*}
&& \big|\E\big[\int_0^1\int_0^1 s\Ll \nabla^2f\big(\omega_0+ss'\Delta\omega_0)\big)-\nabla^2f(\omega_0),\Delta\omega_0\Delta\omega_0^\top     \Rr_{\text {HS}}\dif s\dif s'1_{\{|\Delta\omega_0|\le1\}}\big]\big|\\
&\le&\E\big[\int_0^1\int_0^1 s' s^2 \frac{|\nabla^2f\big(\omega_0+ss'\Delta\omega_0)\big)-\nabla^2f(\omega_0)|}{|ss'\Delta\omega_0|}|\Delta\omega_0|^3 \dif s\dif s'1_{\{|\Delta\omega_0|\le1\}}\big]\\
&\le& C\E\big[(1+|\omega_0+\Delta\omega_0|^5)|\Delta\omega_0|^3 1_{\{|\Delta\omega_0|\le1\}}\big]\\
&\le& C\big(\E\big[(1+|\omega_0+\Delta\omega_0|^5)\ 1_{\{|\Delta\omega_0|\le1\}}\big]^2\big)^{1/2} \big(\E|\Delta\omega_0|^6\big)^{1/2}
\le C\eta^{\frac32}.
\end{eqnarray*}
For the second term, \eqref{e:regularity3} yields
\begin{eqnarray*}
&&\big|\E\big[\int_0^1\int_0^1 s\Ll \nabla^2f\big(\omega_0+ss'\Delta\omega_0)\big)-\nabla^2f(\omega_0),\Delta\omega_0\Delta\omega_0^\top     \Rr_{\text {HS}}\dif s\dif s' 1_{\{|\Delta\omega_0|>1\}}\big]\big|\\
&\le&C\E\big[(1+|\omega_0|^4+|\Delta\omega_0|^4)|\Delta\omega_0|^21_{\{|\Delta\omega_0|>1\}}\big]\\
&\le&C\E\big[(1+|\omega_0|^4+|\Delta\omega_0|^4)^2|\Delta\omega_0|^4\big]\E1_{\{|\Delta\omega_0|>1\}}.
\end{eqnarray*}
By the Markov inequality and \eqref{e:sgld}, we can get
\begin{eqnarray*}
\E1_{\{|\Delta\omega_0|>1\}}=\PP(|\Delta\omega_0|>1)\le \E|\Delta\omega_0|^4\le C\eta^2.
\end{eqnarray*}
Since $\E\big[(1+|\omega_0|^4+|\Delta\omega_0|^4)^2|\Delta\omega_0|^4\big]$ is bounded, we obtain
\begin{eqnarray}\label{e:lem3.1.3}
	\E\big[\int_0^1\int_0^1 s\Ll \nabla^2f\big(\omega_0+ss'\Delta\omega_0)\big)-\nabla^2f(\omega_0),\Delta\omega_0\Delta\omega_0^\top    \Rr_{\text {HS}}\dif s\dif s'\big]
	\le C\eta^{\frac32}.
\end{eqnarray}
For the second term of \eqref{e:lem3.1.2}, similar with the estimation of \eqref{e:lem3.1.3}, \eqref{e:regularity3} implies
\begin{eqnarray}\label{e:lem3.1.4}
\E[\Ll  \nabla^2f(\omega_0),\eta\nabla P(\omega_0)\nabla P(\omega_0)^\top
\Rr_{\text{HS}}]\le C\eta.
\end{eqnarray}
Combining \eqref{e:lem3.1.2} - \eqref{e:lem3.1.4}, we have
\begin{eqnarray*}
W_1(\pi,\pi_\eta)\le C\eta^{1/2}.
\end{eqnarray*}
\end{proof}

\begin{proof}[Proof of Theorem \ref{thm-self-normal}]
According to the decomposition in \eqref{e:Decompose}, we have
\begin{equation*}
\frac{\sqrt{m\eta}}{\sqrt\delta}(\Pi_\eta(h)-\pi(h))= \mathcal{H}_\eta+\mathcal{R}_\eta.
\end{equation*}
Thus for any $x>0$ and $0<y<x$, we have
\begin{eqnarray}\label{e:thm1}
\PP(\mathcal W_\eta> x)\le \PP(\mathcal H_\eta/\sqrt{\mathcal Y_\eta}>x-y)+\PP(\mathcal R_\eta/\sqrt{\mathcal Y_\eta}>y).
\end{eqnarray}
Recall that
\begin{eqnarray*}
	\mathcal{H}_\eta = -\frac{1}{\sqrt{m}}\sum_{k=0}^{m-1}\Ll\nabla f(\omega_k ),\xi_{k+1}\Rr,\quad
		\Y_\eta=\frac{1}{m}\sum_{k=0}^{m-1} |\nabla f(\omega_k)|^2.
\end{eqnarray*}
We denote
\begin{eqnarray*}
	\widehat{\nabla f}(\omega_k)=\nabla f(\omega_k)1_{\{|\omega_k|\le m^{\frac1{12}}\}}\quad 
	\hat\Y_\eta=\frac{1}{m}\sum_{k=0}^{m-1} |\widehat{\nabla f}(\omega_k)|^2.
\end{eqnarray*}

For the probability $\PP(\mathcal H_\eta/\sqrt{\mathcal Y_\eta}>x-y)$ , we have
\begin{eqnarray}\label{e:thm2}
&&\frac{\PP(\mathcal H_\eta/\sqrt{\mathcal Y_\eta}>x-y )}{1-\Phi(x)} \nonumber\\
&\le&\frac{\PP(\mathcal H_\eta/\sqrt{\mathcal Y_\eta}>x-y, ~|\omega_k|\le m^{\frac1{12}}~\text{for any}~  k \in [0, m-1])}{1-\Phi(x-y)}\frac{1-\Phi(x-y)}{1-\Phi(x)} \\
&& +\ \frac{\sum_{k=0}^{m-1}\PP(|\omega_k|>m^{\frac1{12}})}{1-\Phi(x)}.\nonumber
\end{eqnarray}
For the first term above,
\begin{eqnarray*}
&&\frac{\PP(\mathcal H_\eta/\sqrt{\mathcal Y_\eta}>x-y, ~|\omega_k|\le m^{\frac1{12}}~\text{for any}~ k \in [0, m-1])}{1-\Phi(x-y)}\\
&=& \frac{\PP\big(\frac{1}{\sqrt{m \hat \Y_\eta}}\sum_{k=0}^{m-1}\Ll\widehat{\nabla f}(\omega_k ),\xi_{k+1}\Rr
	>x-y, ~|\omega_k|\le m^{\frac1{12}}~\text{for any}~ k \in [0, m-1]\big)}{1-\Phi(x-y)}\\
&\le & \frac{\PP\big(\frac{1}{\sqrt{m \hat \Y_\eta}}\sum_{k=0}^{m-1}\Ll\widehat{\nabla f}(\omega_k ),\xi_{k+1}\Rr
	>x-y\big)}{1-\Phi(x-y)}.
\end{eqnarray*}
It is easy to see that $\Big(\frac{1}{\sqrt m}\Ll \widehat{\nabla f}(\omega_k),\xi_{k+1}\Rr, \mathcal{F}_{k+1} \Big)_{k\ge0}$ is a sequence of martingale difference and
$\sum_{k=0}^{m-1}\E_k[\frac{1}{m}\Ll \widehat{\nabla f}(\omega_k),\xi_{k+1}  \Rr^2]=\hat\Y_\eta$. As $\xi_{k+1}$ is a normal random variable and satisfies the Bernstein condition, Condition (A1) of Lemma \ref{lem:fan} is satisfied. For (A2),
\begin{eqnarray}\label{e:thm4}
	\PP\big( |\hat\Y_\eta-\E\hat\Y_\eta|\ge x' \big)&=& 	\PP\big( |\sum_{k=0}^{m-1}(|\widehat{\nabla f}(\omega_k)|^2-\E|\widehat{\nabla f}(\omega_k)|^2) |\ge mx' \big)\\
	&\le& 2\exp\{-c\, m^{\frac12}x'^2\},\nonumber
\end{eqnarray}
where the last inequality follows \cite[Theorem 2]{dedecker2015subgaussian}. Thus conditions of Lemma \ref{lem:fan} are satisfied with $\e_m=m^{-\frac1{4}}$ and $\delta_m=m^{-\frac1{4}}$ therein. By Lemma \ref{lem:fan}, we obtain for all $0 \le x=o(m^{\frac1{4}})$,
\begin{eqnarray*}
&&\frac{\PP\big(\frac{1}{\sqrt{m \hat \Y_\eta}}\sum_{k=0}^{m-1}\Ll\widehat{\nabla f}(\omega_k ),\xi_{k+1}\Rr>x-y\big)}{1-\Phi(x-y)}\\
&\le& \exp\{C((x-y)^3m^{-1/4}+(1+x-y)m^{-1/4}\ln m)\}.
\end{eqnarray*}
For the tail of the normal distribution, one has the estimation
\begin{eqnarray*}
\frac{1}{\sqrt{2\pi}(1+x)} e^{-\frac{x^2}{2}}\le 1-\Phi(x)\le \frac{1}{\sqrt{\pi}(1+x)} e^{-\frac{x^2}{2}}, \ \ \ x \ge 0,
\end{eqnarray*}
and
\begin{eqnarray*}
\frac{1-\Phi(x-y)}{1-\Phi(x)} = 1+\frac{\int_{x-y}^x e^{-\frac12s^2}\dif s}{\int_x^\infty e^{-\frac12s^2}\dif s}
\le 1+(1+x)ye^{\frac12x^2-\frac12(x-y)^2}
\le e^{Cxy}.
	\end{eqnarray*}
Thus, for the first term of \eqref{e:thm2}, we obtain
for all $0 \le x=o(m^{\frac1{4}})$,
\begin{eqnarray*}
&&\frac{\PP(\mathcal H_\eta/\sqrt{\mathcal Y_\eta}>x-y, ~|\omega_k|\le m^{\frac1{12}}~\text{for any}~ k \in [0, m-1] )}{1-\Phi(x-y)}\frac{1-\Phi(x-y)}{1-\Phi(x)}\\
&\le&\exp\big\{C\big((x-y)^3m^{-1/4}+(1+x-y)m^{-1/4}\ln m+xy\big)\big\}.
\end{eqnarray*}
For the second term of \eqref{e:thm2}, Lemma \ref{lem:expmoment} and the Markov inequality yield that
for all $0\le x =o(m^{\frac{1}{12}}) ,$
\begin{eqnarray*}
	\frac{\sum_{k=0}^{m-1}\PP(|\omega_k|>m^{\frac1{12}})}{1-\Phi(x)}
	&\le& \sum_{k=0}^{m-1} \sqrt{2\pi}(1+x)\E\exp\{C|\omega_k|^2\}e^{-Cm^{\frac1{6}}+x^2/2}\\
	&\le& C   \exp\{-c(m^{\frac1{6}}-x^2)\}.
\end{eqnarray*}
Combining the above estimation for \eqref{e:thm2}, we obtain for all $0 \le x=o(m^{\frac1{12}})$,
\begin{eqnarray}\label{e:thm5}
	&&\frac{\PP(\mathcal H_\eta/\sqrt{\mathcal Y_\eta}>x-y )}{1-\Phi(x)} \nonumber\\
	&\le& \exp\big\{C\big(x^3m^{-1/4}+(1+x)m^{-1/4}\ln m+xy\big)\big\}+ C \exp\{-c(m^{\frac1{6}}-x^2)\}.
\end{eqnarray}
Now we estimate the remainder term $\mathcal{R}_\eta$,
\begin{eqnarray*}
\PP(\mathcal{R}_\eta/\sqrt{\Y_\eta}\ge y)
&\le& \PP(\mathcal{R}_\eta/\sqrt{\Y_\eta}\ge y, \Y_\eta\ge \E\Y_\eta-\frac12\E\Y_\eta)+\PP(\Y_\eta\le \E\Y_\eta-\frac12\E\Y_\eta)\\
&\le& \PP(\mathcal{R}_\eta\ge y\sqrt{\E\Y_\eta/2})+\PP( \E\Y_\eta-\Y_\eta\ge\frac12\E\Y_\eta).
\end{eqnarray*}
According to Lemma \ref{lem:R}, we have
\begin{eqnarray*}
\PP(\mathcal R_\eta\ge y\sqrt{\E\Y_\eta/2})\le C\Big(e^{-c\eta^{-\frac15}\delta^{\frac15}y^{\frac25}}1_{\{y< m^{-\frac56}\eta^{-\frac76}\delta^{-\frac12}\}}+e^{-cm^{\frac12}\eta^{\frac12}\delta^{\frac12}y}1_{\{y\ge m^{-\frac56}\eta^{-\frac76}\delta^{-\frac12}\}}\Big),
\end{eqnarray*}
for $c_{m,\eta}\le y$. Similar with the calculation of \eqref{e:thm4}, we can get
\begin{eqnarray*}
\PP( \E\Y_\eta-\Y_\eta\ge\E\Y_\eta/2)\le \PP\big( \E\hat\Y_\eta-\hat\Y_\eta\ge \E\Y_\eta/2 \big)+\sum_{k=0}^{m-1}\PP(|\omega_k|\ge m^{\frac1{12}})\le Ce^{-cm^{\frac16}}.
\end{eqnarray*}
This yields
\begin{eqnarray}\label{e:thm6}
& & \nonumber \frac{\PP(\mathcal R_\eta\ge y\sqrt{\E\Y_\eta/2})}{1-\Phi(x)}  \\
& \le&  C\big(\exp\{-c(m^{\frac16}-x^2)\}+\exp\{-c\eta^{-\frac15}\delta^{\frac15}y^{\frac25}+x^2\}1_{\{c_{m,\eta}\le y\le m^{-\frac56}\eta^{-\frac76}\delta^{-\frac12}\}}\\
& &+\exp\{-cm^{\frac12}\eta^{\frac12}\delta^{\frac12}y+x^2\}1_{\{y\ge m^{-\frac56}\eta^{-\frac76}\delta^{-\frac12}\}}\big)\nonumber.
\end{eqnarray}

For the case $m\le\eta^{-\frac{13}8}\delta^{-\frac98}$, combing \eqref{e:thm1}, \eqref{e:thm5} and \eqref{e:thm6} with $y=x^5\eta^{\frac12}\delta^{-\frac12}+\eta^{\frac12}\delta^{-\frac12}|\ln\eta|$, we have that
\begin{eqnarray*}
	\frac{\PP(\mathcal W_\eta> x)}{1-\Phi(x)}
	&\le& \exp\big\{C\big(x^3m^{-\frac14}+(1+x)m^{-\frac14}\ln m+x^6\eta^{\frac12}\delta^{-\frac12}+x\eta^{\frac12}\delta^{-\frac12}|\ln\eta|\big)\big\}\\
	&&+ C\big(\exp\{-c(x^5+|\ln\eta|)^{\frac25}\}1_{\{0\le x<m^{-\frac16}\eta^{-\frac13}\}}+\exp\{-c(m^{\frac16}-x^2)\}\\
	&&\quad\quad+\exp\{-c(x^5\eta m^{\frac12})\}1_{\{x\ge m^{-\frac16}\eta^{-\frac13}\}}\big)\\
	&\le& 1+ C\big(x^3m^{-\frac14}+(1+x)m^{-\frac14}\ln m+x^6\eta^{\frac12}\delta^{-\frac12}\big)
\end{eqnarray*}
holds uniformly for $0\le x =o(\eta^{-\frac1{12}}\delta^{\frac1{12}})$. On the other hand, we can get similarly,
\begin{eqnarray*}
	\frac{\PP(\mathcal W_\eta> x)}{1-\Phi(x)}
	&\ge& 
	 \PP(\mathcal H_\eta/\sqrt{\mathcal Y_\eta}>x+y)-\PP(\mathcal R_\eta/\sqrt{\mathcal Y_\eta}<-y)\\
	&\ge& 1- C\big(x^3m^{-\frac14}+(1+x)m^{-\frac14}\ln m+x^6\eta^{\frac12}\delta^{-\frac12}\big).
\end{eqnarray*}
Thus, we can get
\begin{eqnarray*}
\Big| \frac{\PP(\mathcal W_\eta> x)}{1-\Phi(x)}
	-1 \Big| & \le & 
C\big(x^3m^{-\frac14}+(1+x)m^{-\frac14}\ln m+x^6\eta^{\frac12}\delta^{-\frac12}\big)
\end{eqnarray*}
uniformly for $0\le x=o(\eta^{-\frac1{12}}\delta^{\frac1{12}})$ as $\eta$ tends to zero and $m$ tends to infinity.  

For the case $m>\eta^{-\frac{13}8}\delta^{-\frac98}$, it is easy to verify that $c_{m,\eta}\ge m^{-\frac56}\eta^{-\frac76}\delta^{-\frac12}$. Combing \eqref{e:thm1}, \eqref{e:thm5} and \eqref{e:thm6} with $y=x^2(m\eta\delta)^{-\frac12}+\sqrt m\eta\delta$, we can get
\begin{eqnarray*}
	\frac{\PP(\mathcal W_\eta> x)}{1-\Phi(x)}
	&\le& \exp\big\{C\big(x^3m^{-\frac14}+(1+x)m^{-\frac14}\ln m+xy\big)\big\}\\
	&&+ C\big(\exp\{-cm^{\frac12}\eta^{\frac12}\delta^{\frac12}y+x^2\}+\exp\{-c(m^{\frac16}-x^2)\}\big)\\
	&\le& 1+  C\big(x^3(m\eta\delta)^{-\frac12}+\sqrt m\eta\delta x+ m^{-\frac14}\ln m\big),
\end{eqnarray*}
holds uniformly for $0\le x =o( (m\eta\delta)^{\frac16}\wedge (\sqrt m\eta\delta)^{-1})$. On the other hand, using similar arguments we have,
\begin{eqnarray*}
	\frac{\PP(\mathcal W_\eta> x)}{1-\Phi(x)}
	&\ge& 
	\PP(\mathcal H_\eta/\sqrt{\mathcal Y_\eta}>x+y)-\PP(\mathcal R_\eta/\sqrt{\mathcal Y_\eta}<-y)\\
	&\ge& 1+  C\big(x^3(m\eta\delta)^{-\frac12}+\sqrt m\eta\delta x+ m^{-\frac14}\ln m\big).
\end{eqnarray*}
Thus,
\begin{eqnarray*}
	\Big| \frac{\PP(\mathcal W_\eta> x)}{1-\Phi(x)}
	-1 \Big| & \le & 
	 C\big(x^3(m\eta\delta)^{-\frac12}+\sqrt m\eta\delta x+ m^{-\frac14}\ln m\big)
\end{eqnarray*}
uniformly for $0\le x=o( (m\eta\delta)^{\frac16}\wedge (\sqrt m\eta\delta)^{-1})$ as $\eta$ tends to zero and $m$ tends to infinity. 

\end{proof}

\begin{proof}[Proof of Theorem \ref{coro:berry}]
For the case $m\le\eta^{-\frac{13}8}\delta^{-\frac98}$, denote $C_{m,\eta}=\eta^{{-\frac{1}{24}}}\delta^{\frac{1}{24}} $. It is easy to have the following decomposition,	
\begin{eqnarray*}
	\sup_{x\in\R}|\PP(\mathcal{W}_\eta< x)-\Phi(x)|&\le& 	\sup_{x\le -C_{m,\eta}}|\PP(\mathcal{W}_\eta\le x)-\Phi(x)|+\sup_{-C_{m,\eta}\le x\le 0 }|\PP(\mathcal{W}_\eta\le x)-\Phi(x)|\\
	&&+\sup_{0\le x\le C_{m,\eta}}|\PP(\mathcal{W}_\eta\le x)-\Phi(x)|+\sup_{x> C_{m,\eta}}|\PP(\mathcal{W}_\eta\le x)-\Phi(x)|\\
	&=:& I_1+I_2+I_3+I_4.
\end{eqnarray*}

For $I_1$ and $I_4$,  Theorem \ref{thm-self-normal} implies
\begin{eqnarray*}
I_1&=& \sup_{x\le -C_{m,\eta}}|\PP(\mathcal{W}_\eta\le x)-\Phi(x)|\\
&\le& \sup_{x\le -C_{m,\eta}}\PP(\mathcal{W}_\eta\le x)+\Phi(-c_{\eta,m})\\
&\le& \Phi(-C_{m,\eta})e^C+\Phi(-C_{m,\eta})\le Cm^{-\frac14}\ln m.
\end{eqnarray*}
Similarly,
\begin{eqnarray*}
I_4\le Cm^{-\frac14}\ln m.
\end{eqnarray*}
For $I_2$ and $I_3$, Theorem \ref{thm-self-normal} and the inequality $|e^x-1|\le|x|e^{|x|}$ imply
	\begin{eqnarray*}
	I_2&=&\sup_{-C_{m,\eta}\le x\le 0 }|\PP(\mathcal{W}_\eta\le x)-\Phi(x)|\\
	&\le&\sup_{-C_{m,\eta}\le x\le 0 }C\Phi(x)\big(x^3m^{-\frac14}+(1+x)m^{-\frac14}\ln m+x^6\eta^{\frac12}\delta^{-\frac12}\big)\\
	&\le& Cm^{-\frac14}\ln m.
	\end{eqnarray*}
	Similarly,
\begin{eqnarray*}
		I_3\le Cm^{-\frac14}\ln m.
\end{eqnarray*}
Combining the estimation for the terms $I_1$-$I_4$, we have
\begin{eqnarray*}
	\sup_{x\in\R}|\PP(\mathcal{W}_\eta< x)-\Phi(x)|&\le& 	Cm^{-\frac14}\ln m.
\end{eqnarray*}		

For the case $m>\eta^{-\frac{13}8}\delta^{-\frac98}$, taking $C_{m,\eta}= (m\eta\delta)^{\frac1{12}}\wedge (\sqrt m\eta\delta)^{-\frac12}$ instead of $m^{1/24}$, we can similarly show that
\begin{eqnarray*}
	\sup_{x\in\R}|\PP(\mathcal{W}_\eta< x)-\Phi(x)|&\le& Cm^{-\frac14}\ln m.
\end{eqnarray*}		

Thus, for any $\eta^{-1}<m=o(\eta^{-2})$, we have
\begin{eqnarray*}
	\sup_{x\in\R}|\PP(\mathcal{W}_\eta< x)-\Phi(x)|&\le& Cm^{-\frac14}\ln m.
\end{eqnarray*}		
Further assume $m=C\eta^{-2}/|\ln\eta|$, we can get
\begin{eqnarray*}
	\sup_{x\in\R}|\PP(\mathcal{W}_\eta\le x)-\Phi(x)|
	&\le& C\eta^{\frac12}|\ln \eta|^{5/4}.
\end{eqnarray*}	
\end{proof}	

\section{Proof of Lemma \ref{lem:stein}}\label{appen:regularity}

The proof of Lemma \ref{lem:stein} follows from \cite[Corollary 6.3]{gilbarg2001elliptic}. For ease of reading, their result is given below. Let $\Omega$ be an open subset of $\R^d$, $\alpha\in(0,1]$. For any function defined on $\R^d$, denote
\begin{eqnarray*}
	\|f\|_{0;\Omega}=\sup_{x\in\Omega}\|f(x)\|,
	\quad [f]_{\alpha;\Omega}=\sup_{x,y\in\Omega, x\neq y}\frac{\|f(x)-f(y)\|}{\|x-y\|^\alpha},
	\quad \|f\|_{0, \alpha;\Omega}=\|f\|_{0;\Omega}+[f]_{\alpha;\Omega}.
\end{eqnarray*}
Let $C^{k}(\R^d)$, where $k \ge1$, denote the collection of all $k$-th order continuously differentiable functions on $\R^d$. $C^{k,\alpha}(\R^d)$, with $\alpha\in(0,1]$, refers to the collection of $k$-th order continuously differentiable functions whose $k$-th order partial derivatives are $\alpha$-H\"older continuous. For the case $k=1$, we simplify the notation to $C^{\alpha}(\R^d)$.

\begin{lemma}{\cite[Corollary 6.3]{gilbarg2001elliptic}}\label{lem:gilbarg}
	Let $f\in C^{2,\alpha}(\Omega)$, $h \in C^{\alpha}(\bar{\Omega})$ satisfy $\mathcal{L}f=h$ in a bounded domain $\Omega$ where 
	\begin{equation*}
		\mathcal{L}f(x)=\Ll a(x), \nabla^2 f(x) \Rr_\mathrm{HS}+\Ll b(x),\nabla f(x)\Rr,
	\end{equation*}
	is strictly elliptic and its coefficients are in $ C^{\alpha}(\bar{\Omega})$. Then if $ \Omega' \subset \subset \Omega$ with $\text{dist}(\Omega', \partial \Omega) \geq \bar d$, there is a constant $ C $ such that
	\begin{equation}\label{e:regularity5}
		\bar d \| \nabla f \|_{0; \Omega'} + \bar d^2 \| \nabla^2 f \|_{0; \Omega'} + \bar d^{2+\alpha} [ \nabla^2 f ]_{\alpha; \Omega'} \leq C ( \| f \|_{0; \Omega} + \|h \|_{0, \alpha; \Omega}),
	\end{equation}
	where the positive constant  $C$  depends only on the ellipticity constant and the $ C^{\alpha}(\bar{\Omega}) $ norms of the coefficients of $\mathcal{L}$.
\end{lemma}

\begin{proof}[Proof of Lemma \ref{lem:stein}]
	The existence and the expression of the solution $f$ can be proved similarly as in \cite[Proposition 6.1]{FSX19}. Now we show the regularities of it. According to \eqref{e:steinsol},
	\begin{eqnarray*}
		|f(x)|\le \int_0^\infty |\E[h(X_t^x)]-\pi(h)|\dif t\le \int_0^\infty C(1+|x|^2)e^{-ct}\dif t\le C(1+|x|^2),
	\end{eqnarray*}	
	where the second inequality follows \eqref{e:ergod1}.
	
	For any $x\in\R^d$, define $r(x)=\frac{1}{2(1+|x|)}\in(0,\frac12]$ and 
	$$B_{r(x)}(x)=\{z\in\R^d: |x-z|\le r(x)\}.$$ 
	Consider $\Omega=B_{r(y)}(y)$ and  $\Omega'=B_{r(y)/2}(y)$ for any $y\in \R^d$ in Lemma \ref{lem:gilbarg}. Then we have $$dist(\Omega',\partial\Omega)\ge \frac{r(y)}{2}=\frac{1}{4(1+|y|)}.$$
Therefore, we take $\bar d=\frac{1}{4(1+|y|)}$.
	Taking $\alpha=1$ in Lemma \ref{lem:gilbarg} and considering operator \eqref{e:generator}, 
	\begin{equation*}
		\mathcal L f=\Ll -\nabla P, \nabla f \Rr+\frac1 2\Ll Q_{\eta,\delta},\nabla^2f \Rr_{\text{HS}}. 
	\end{equation*}
	The notation of $Q_{\eta,\delta}$  implies that $\mathcal{L}$ is strictly elliptic, thus \eqref{e:plip} and \eqref{e:qlip} yield that its coefficients are Lipschitz functions in $\bar\Omega$ which satisfy the condition of Lemma \ref{lem:gilbarg}. Then we have
	\begin{eqnarray}
		&&r(y) \| \nabla f \|_{0; \Omega'} \leq C ( \| f \|_{0; \Omega} + \|h \|_{0, 1; \Omega}),\label{e:regularity6}\\
		&&r(y)^2 \| \nabla^2 f \|_{0; \Omega'} \leq C ( \| f \|_{0; \Omega} + \|h \|_{0, 1; \Omega}),\label{e:regularity7}\\ 
		&&r(y)^3 [ \nabla^2 f ]_{1; \Omega'} \leq C ( \| f \|_{0; \Omega} + \|h \|_{0, 1; \Omega}).\label{e:regularity8}
	\end{eqnarray}
	
	For the equality \eqref{e:regularity2}, since 
	$$\int_{\R^d}r(x)\dif x=\infty,
	$$
	for any $0<r_0\le1$, we have 
	$$
	B_{r_0}(x)\subset \bigcup_{y\in B_{r_0(x)}} B_{r_{(y)/2}}(y)=\bigcup_{y\in B_{r_0(x)}} \Omega'.
	$$
	Combining with \eqref{e:regularity6}, we obtain
	\begin{eqnarray*}
		\| \nabla f \|_{0; B_{r_0}(x)}
		\le \sup_{y\in B_{r_0}(x)}  \| \nabla f \|_{0; B_{r(y)/2}(y)}\le \sup_{y\in B_{r_0}(x)} C(1+|y|)( \| f \|_{0; \Omega} + \|h \|_{0, 1; \Omega}).
	\end{eqnarray*}
	\eqref{e:regularity1} implies that
	\begin{eqnarray*}
		\| f \|_{0; \Omega}\le \sup_{z\in B_{r(y)}(y)}|f(z)|\le \sup_{z\in B_{r(y)}(y)}C(1+|z|^2)\le C(1+|y|^2).
	\end{eqnarray*}
	Since $h$ is the Lipschitz function, then we have
	\begin{eqnarray*}
		\| h \|_{0,1; \Omega}\le \sup_{z\in B_{r(y)}(y)}|h(z)|+\sup_{z_1,z_2\in B_{r(y)}(y)}\frac{|h(z_1)-h(z_2)|}{|z_1-z_2|}\le C(1+|y|^2).
	\end{eqnarray*}
	Thus,
	\begin{eqnarray*}
		\| \nabla f \|_{0; B_{r_0}(x)}
		\le \sup_{y\in B_{r_0}(x)} C(1+|y|)( 1+|y^2|)\le C(1+|x|^3),
	\end{eqnarray*}
	which yields
	\begin{eqnarray*}
		| \nabla f(x) | \le C(1+|x|^3).
	\end{eqnarray*}
	
	Similarly, \eqref{e:regularity7} and \eqref{e:regularity8} imply
	\begin{eqnarray*}
		\| \nabla^2 f \|_{0; B_{r_0}(x)} \le C(1+|x|^4),
	\end{eqnarray*}
	\begin{eqnarray*}
		[\nabla^2 f ]_{1; B_{r_0}(x)} \le C(1+|x|^5).
	\end{eqnarray*}
	Thus, we can obtain \eqref{e:regularity3} and \eqref{e:regularity4}.
\end{proof}

\section{Estimation of the remainder $\mcl R_\eta$}\label{appen:3}
We will give in this section several lemmas of $\mcl R_\eta$ which play a crucial role in proving the main results. In order to estimate the tail probability of $\mathcal{R}_\eta$, we need the following four lemmas, the first three lemmas paving the way for proving Lemma \ref{lem:R}.

\begin{lemma}\label{lem:expmoment}
	For small enough $\gamma>0$, one has
	\begin{eqnarray*}
		\E[\exp\{\gamma|\omega_{k}|^2\}\le C,
	\end{eqnarray*}
	for any $k$.
\end{lemma}
\begin{proof}

For small enough $\gamma>0$ and any constant $k$, \eqref{e:sgld} implies
\begin{eqnarray*}
	\E[\exp\{\gamma|\omega_{k+1}|^2\}
	&=& \E\left[\exp\big\{\gamma(|\omega_k|^2+|\eta \nabla \psi(\omega_k, \zeta_{k+1})|^2+2\Ll\omega_k,-\eta \nabla \psi(\omega_k, \zeta_{k+1})\Rr)\big\}\right. \\
	&& \left.\quad \E_k[\exp \{\eta \delta \gamma|\xi_{k+1}|^2+2\gamma\langle\omega_k-\eta \nabla \psi(\omega_k, \zeta_{k+1}), \sqrt{\eta \delta} \xi_{k+1}\rangle \}|\zeta_{k+1}]\right]
\end{eqnarray*}
A straight calculation to the conditional expectation with respect to the Gaussian random variable $\xi_{k+1}$, we have
\begin{eqnarray*}
	&&\E_k[\exp \{\eta \delta \gamma|\xi_{k+1}|^2+2\gamma\langle\omega_k-\eta \nabla \psi(\omega_k, \zeta_{k+1}), \sqrt{\eta \delta} \xi_{k+1}\rangle \}| \zeta_{k+1}]\\
	&=& \frac{1}{\sqrt{1-2 \eta \delta \gamma}}  \exp \big\{ \frac{2\eta \delta \gamma^2}{1-2 \eta\delta\gamma}\left|\omega_k-\eta \nabla \psi\left(\omega_k, \zeta_{k+1}\right)\right|^2 \big\}\\
	&\le& \frac{1}{\sqrt{1-2 \eta \delta \gamma}}  \exp \big\{ \frac{4\eta \delta \gamma^2}{1-2 \eta\delta\gamma}(|\omega_k|^2+\eta^2 |\nabla \psi\left(\omega_k, \zeta_{k+1}\right)|^2) \big\}.
\end{eqnarray*}
Here $\gamma$ is chosen to be small enough that $1-2\eta\delta\gamma>0$.  Then one has,
\begin{eqnarray*}
	&&\E[\exp\{\gamma|\omega_{k+1}|^2\}\\
	&\le&\frac{1}{\sqrt{1-2\eta\delta\gamma}}\E\Big[\exp \big\{(1+\frac{4 \eta \delta \gamma}{1-2 \eta \delta \gamma}) \gamma|\omega_k|^2+(1+\frac{4 \eta \delta \gamma}{1-2 \eta \delta \gamma}) \gamma \eta^2 |\nabla \psi(\omega_k, \zeta_{k+1})|^2\\
	&&\quad+2 \gamma\Ll \omega_k,-\eta \nabla_k(\omega_k,\zeta_{k+1})\big\}\Big]\\
	&\le& \frac{\exp\{2\gamma\eta K_2\}}{\sqrt{1-2\eta\delta\gamma}}\E\Big[\exp \big\{\big(1+\frac{4 \eta \delta \gamma}{1-2 \eta \delta \gamma}+2(1+\frac{4 \eta \delta \gamma}{1-2 \eta \delta \gamma})\eta^2L^2-K_1\eta\big) \gamma|\omega_k|^2\\
	&&\quad+\big(2\gamma\eta^2(1+\frac{4 \eta \delta \gamma}{1-2 \eta \delta \gamma})+\frac{\gamma\eta}{K_1}\big)|\nabla \psi(0, \zeta_{k+1})|^2\big\}\Big].
\end{eqnarray*}
Since $\omega_k$ and $\zeta_{k+1}$ are independent, and $\nabla\psi(0,\zeta_{k+1})$ is sub-Gaussian, we can choose small enough $\gamma$ such that
\begin{eqnarray*}
	\E[\exp\{\gamma|\omega_{k+1}|^2\}
	&\le& \frac{\exp\{2\gamma\eta K_2\}}{\sqrt{1-2\eta\delta\gamma}}\E\big[\exp \big\{(1-\frac12K_1\eta)\gamma|\omega_k|^2\big\}\big] \\
	&&\times \Big( \E\big[\exp\big\{ (2\gamma\eta(1+\frac{4\eta\delta\gamma}{1-2\eta\delta\gamma})+\frac{\gamma}{k_1})|\nabla\psi(0,\zeta_{k+1})|^2   \big\}\big]  \Big)^\eta\\
	&\le& \frac{C^\eta}{\sqrt{1-2\eta\delta\gamma}}\big(\E\exp \{\gamma|\omega_k|^2\}\big)^{1-\frac12K_1\eta},
\end{eqnarray*}
where the last line follows the H\"older inequality. Inductively, we can get
\begin{eqnarray*}
	\E[\exp\{\gamma|\omega_{k+1}|^2\} &\le& \frac{C^\eta}{\sqrt{1-2\eta\delta\gamma}}\big(\E\exp \{\gamma|\omega_k|^2\}\big)^{1-\frac12K_1\eta}\\
	&\le& \Big( \frac{C^\eta}{\sqrt{1-2\eta\delta\gamma}} \Big)^\frac{c}{\eta}\big(\E\exp\{\gamma|\omega_0|^2\}\big)^{(1-K_1\eta/2)^{k+1}}\le C.
\end{eqnarray*}
Thus the exponential moment of $|\omega_k|^2$ exists for any $k$ and small enough $\gamma$.
\end{proof}

\begin{lemma}\label{lem:RC}
	Let Assumptions \ref{assu1} and \ref{assu2} hold, considering the martingale difference\\ $(\Psi(\omega_k,\theta_{k+1}), \mathcal F_{k+1})_{k\ge0}$ with
	\begin{eqnarray}\label{e:RC}
		\E_k|\Psi(\omega_k,\theta_{k+1})|^i\le 		C^i(1+|\omega_k|^{\alpha i}+i!),
	\end{eqnarray}
	where $\alpha\ge0$ and $i$ is any positive integer. Then for $\sqrt m=o(x)$, we have
	\begin{eqnarray}\label{e:RC1}
		\PP\Big(\sum_{k=0}^{m-1} \Ll \nabla f(\omega_k), \Psi(\omega_k,\theta_{k+1}) \Rr >x\Big)\le C\exp\{c(x^2/m)^{\frac1{4+\alpha}}\}.
	\end{eqnarray}
	Similarly,
	\begin{eqnarray}\label{e:RC2}
		\PP\Big(\sum_{k=0}^{m-1} \Ll \nabla^2 f(\omega_k), \Psi(\omega_k,\theta_{k+1}) \Rr_{\mathrm{HS}} >x\Big)\le C\exp\{c(x^2/m)^{\frac1{5+\alpha}}\}.
	\end{eqnarray}
	Here $f$ is the solution of Stein's equation given in Lemma \ref{lem:stein}.
	
\end{lemma}	
\begin{proof}
	Denote $\hat\omega_k=\omega_k1_{\{|\omega_k|\le y\}}$ for large enough $y$ to be chosen later, $A=\{|\omega_k|\le y, k=0,1,...,m-1\}$ and $A^C$ as its complement, then we have
	\begin{align}\label{e:eqR1}
		&\PP\Big(\sum_{k=0}^{m-1} \Ll \nabla f(\omega_k), \Psi(\omega_k,\theta_{k+1}) \Rr >x\Big)\nonumber\\
		\le& 		\PP\Big(\sum_{k=0}^{m-1} \Ll \nabla f(\omega_k), \Psi(\omega_k,\theta_{k+1}) \Rr >x, A\Big)+\PP(A^C)\nonumber\\
		\le& 		\PP\Big(\sum_{k=0}^{m-1} \Ll \nabla f(\hat\omega_k), \Psi(\hat\omega_k,\theta_{k+1}) \Rr >x\Big)+\sum_{k=0}^{m-1}\PP\big(|\omega_k|>y\big)\nonumber\\
		\le& e^{-\lambda x}\E\exp\Big\{\sum_{k=0}^{m-1}\lambda\Ll \nabla f(\hat\omega_k),\Psi(\hat\omega_k,\theta_{k+1})\Rr\Big\} +e^{-\gamma y^2}\sum_{k=0}^{m-1}\E\exp\{\gamma|\omega_k|^2\},
	\end{align}
	where the last inequality follows from the Markov inequality, $\lambda$ is a positive constant to be chosen later, and $\gamma$ is a sufficiently small positive constant. For the second term of \eqref{e:eqR1}, Lemma \ref{lem:expmoment} implies
	\begin{eqnarray*}
		e^{-\gamma y^2}\sum_{k=0}^{m-1}\E\exp\{\gamma|\omega_k|^2\}\le mCe^{-\gamma y^2}. 
	\end{eqnarray*}
	For the first term of \eqref{e:eqR1}, it is easy to see that,
	\begin{eqnarray*}
		&&\E\exp\Big\{\sum_{k=0}^{m-1}\lambda\Ll \nabla f(\hat\omega_k),\Psi(\hat\omega_k,\theta_{k+1})\Rr\Big\} \\
		&=&\E\Big[\exp\big\{\sum_{k=0}^{m-2}\lambda\Ll \nabla f(\hat\omega_k),\Psi(\hat\omega_k,\theta_{k+1})\Rr\big\} \E_{m-1}\exp\{\lambda\Ll \nabla f(\hat\omega_{m-1}),\Psi(\hat\omega_{m-1},\theta_{m})\Rr\}\Big].
	\end{eqnarray*}
	Noticing 
	$$\lambda\E_{m-1}\Ll \nabla f(\hat\omega_{m-1}),\Psi(\hat\omega_{m-1},\theta_{m})\Rr=0,$$ 
	by the Taylor expansion to the conditional expectation above, \eqref{e:regularity2}  and \eqref{e:RC} imply
	\begin{eqnarray*}
		&&\E_{m-1}\exp\{\lambda\E_{m-1}\Ll \nabla f(\hat\omega_{m-1}),\Psi(\hat\omega_k,\theta_{m})\Rr\}\\
		&=&1+\sum_{i=2}^{\infty}\frac{\lambda^i}{i!}\E_{m-1}\Ll \nabla f(\hat\omega_{m-1}), 
		\Psi(\hat\omega_k,\theta_{m})\Rr^i\\
		&\le&1+\sum_{i=2}^\infty\frac{(C\lambda)^i}{i!}(1+y^3)^i (1+y^{\alpha i}+i!)\\
		&\le&1+\frac{(C\lambda y^{3+\alpha})^2}{1-C\lambda y^{3+\alpha}},
	\end{eqnarray*}
	if $C\lambda y^{3+\alpha}<1$. By induction, we can get
	\begin{eqnarray*}
		\E\exp\Big\{\sum_{k=0}^{m-1}\lambda\Ll \nabla f(\hat\omega_k),\Psi(\hat\omega_k,\theta_{k+1})\Rr\Big\}\le \Big( 1+\frac{(C\lambda y^{3+\alpha})^2}{1-C\lambda y^{3+\alpha}}\Big)^m.
	\end{eqnarray*}	
	Thus for \eqref{e:eqR1}, we have
	\begin{eqnarray*}
		\PP\Big(\sum_{k=0}^{m-1} \Ll \nabla f(\omega_k), \Psi(\omega_k,\theta_{k+1}) \Rr >x\Big)\le \Big( 1+\frac{(C\lambda y^{3+\alpha})^2}{1-C\lambda y^{3+\alpha}}\Big)^m e^{-\lambda x}   + mCe^{-\gamma y^2}.
	\end{eqnarray*}
	Let $\lambda=\frac{x}{2mC^2y^{6+2\alpha}}$ and $y=(\frac{x^2}{2mC^2})^{\frac{1}{8+2\alpha}}$,
	for large enough $m$ and $\sqrt m=o(x)$, one can get
	\begin{eqnarray*}
		\PP\Big(\sum_{k=0}^{m-1} \Ll \nabla f(\omega_k), \Psi(\omega_k,\theta_{k+1}) \Rr >x\Big)\le  Ce^{-c(x^2/m)^{\frac{1}{4+\alpha}}}.
	\end{eqnarray*}
	
	Similarly, we can show \eqref{e:RC2}. The details are  omitted here.
\end{proof}

\begin{proof}[Proof of Lemma \ref{lem:R} ]
	Recalling the definition of $\mathcal{R}_\eta$, we have
	\begin{eqnarray*}
		\PP(|\mathcal{R}_\eta|>y)&\le&\sum_{i=1}^4\PP(|\mathcal{R}_{\eta,i}|>\frac{y}{4}),
	\end{eqnarray*}
	and shall prove below that the following estimates hold:
	\begin{eqnarray}
		\PP(|\mathcal{R}_{\eta,1}|>y/4)&\le &C e^{-c \sqrt{m\eta\delta}y}, \label{e:R-1} \\
		\PP\left(\left\vert\mathcal{R}_{\eta,2}\right\vert>y/4\right) &\le & Ce^{-c\delta^\frac15 y^\frac25\eta^{-\frac1{5}}}, \label{e:R-2} \\
		\PP\left(\left\vert\mathcal{R}_{\eta,3}\right\vert >y/4\right)& \le & Ce^{-cy^{\frac29}\eta^{-\frac29}\delta^{-\frac29} }, \label{e:R-3}\\
		\PP\left(\left\vert\mathcal{R}_{\eta,4}\right\vert >y/4\right)& \le &	Ce^{-cy^{\frac27}\delta^{\frac17}\eta^{-\frac37}}+Ce^{-cy^{\frac25}\delta^{-\frac15}\eta^{-\frac15}}. \label{e:R-4}
	\end{eqnarray}
	Combining these estimates, we immediately get
	\begin{equation*}
		\PP(|\mathcal{R}_\eta|>y)\le C\Big( e^{-cy\eta^{\frac12}\delta^\frac12m^{\frac12} }
		+e^{-cy^\frac25\delta^\frac15 \eta^{-\frac1{5}}}
		+e^{-cy^{\frac16}\eta^{-\frac5{12}}\delta^{-\frac5{12}} }
		+e^{-cy^{\frac27}\delta^{\frac17}\eta^{-\frac37}}\Big),
	\end{equation*}
	for $c(\eta^{\frac12}\delta^{-\frac12}\vee m^{\frac12}\eta\delta)\le y\le C\eta^{-\frac72}\delta^{-\frac72}$. Now we show \eqref{e:R-1}-\eqref{e:R-4} below. 
	
	\textit{a) Control of $\mathcal{R}_{\eta,1}$.}
	By the Markov inequality and \eqref{e:regularity1},
	\begin{eqnarray*}
		\PP(|\mathcal{R}_{\eta,1}|>y/4)&=&\PP(\gamma|f(\omega_0)-f(\omega_m)|>\gamma \sqrt{m\eta\delta}y/4)\\
		&\le& \E\exp\{C\gamma(1+|\omega_0|^2+|\omega_m|^2)\}e^{-\gamma \sqrt{m\eta\delta}y/4}\\
		&\le& (\E\exp\{2C\gamma|\omega_0|^2\})^{1/2}(\E\exp\{2C\gamma|\omega_m|^2\})^{1/2}e^{-\gamma\sqrt{m\eta\delta}y/4+C\gamma},
	\end{eqnarray*}	
	where $\gamma$ is a positive constant. Lemma \ref{lem:expmoment} implies that the exponential moments of $\omega_0$ and $\omega_m$ are finite for small enough $\gamma$. Thus
	\begin{eqnarray*}
		\PP(|\mathcal{R}_{\eta,1}|>y/4)
		&\le& C e^{-c \sqrt{m\eta\delta}y}.
	\end{eqnarray*}	
	
	\textit{b) Control of $\mathcal{R}_{\eta,2}$.}
	According to the definition of $\mathcal{R}_{\eta,2}$, we have
	\begin{eqnarray*}
		\PP\left(\mathcal{R}_{\eta,2}>y/4\right)	
		= \PP\Big(  \sum_{k=0}^{m-1}\Ll \nabla f(\omega_k),\nabla P(\omega_k)-\nabla\psi(\omega_k,\zeta_{k+1})\Rr >\frac{\sqrt{m\delta}y}{4\sqrt\eta} \Big).
	\end{eqnarray*}
	Since $\nabla\psi(0,\zeta_{k+1})$ is sub-Gaussian from Assumption \ref{assu2}, we have
	$$ \E|\nabla\psi(0,\zeta_{k+1})|^i\le Ci!.
	$$
	By \eqref{e:psibound1} and \eqref{e:pbound1}, we have	
	\begin{eqnarray*}
		\E_{k}|\nabla P(\omega_{k})-\nabla\psi(\omega_{k},\zeta_{k+1})|^i
		&\le& \E_k\big[2L|\omega_k|+|\nabla P(0)|+|\nabla \psi(0,\zeta_{k+1})| \big]^i\\
		&\le& C^i(1+|\omega_k|^i+i!),
	\end{eqnarray*}
	which satisfies the condition of Lemma \ref{lem:RC} with $\alpha=1$. Thus, \eqref{e:RC1} yields
	\begin{eqnarray*}
		\PP\left(\mathcal{R}_{\eta,2}>y/4\right)	\le C\exp\{-c\delta^\frac15 y^\frac25\eta^{-\frac1{5}}\},
	\end{eqnarray*}
	under the condition $y>\sqrt{\eta/\delta}$.
	$\PP\left(\mathcal{R}_{\eta,2}<-y/4\right)$ can be estimated similarly. Thus \eqref{e:R-2} is proved.
	
	\textit{c) Control of $\mathcal{R}_{\eta,3}$.} Let $A=\{|\Delta\omega_k|<y_1\le1, k=0,1,...,m-1\}$, we have
	\begin{eqnarray*}
		&&\PP(\mathcal{R}_{\eta,3}>y/4)\\
		&=&\PP\Big(\sum_{k=0}^{m-1}
		\int_0^1\int_0^1s \big\Ll \frac{\nabla^2 f(\omega_k+ss'\Delta\omega_k)-\nabla^2f(\omega_k )}{|ss'\Delta\omega_k|}, \Delta \omega_k\Delta \omega_k^\top\big\Rr_\mathrm{HS}|ss'\Delta\omega_k|\dif s'\dif s>\frac{\sqrt{m\eta\delta}y}4
		\Big)\\
		&\le& \PP\Big(\sum_{k=0}^{m-1}
		\int_0^1\int_0^1  \frac{|\nabla^2 f(\omega_k+ss'\Delta\omega_k)-\nabla^2f(\omega_k )|}{|ss'\Delta\omega_k|} |\Delta \omega_k|^3\dif s'\dif s>\frac{\sqrt{m\eta\delta}y}4, A
		\Big)+\PP(A^C)\\
	\end{eqnarray*}
	For the first term, \eqref{e:regularity4} implies
			\begin{eqnarray*}
		&&\PP\Big(\sum_{k=0}^{m-1}
		\int_0^1\int_0^1  \frac{|\nabla^2 f(\omega_k+ss'\Delta\omega_k)-\nabla^2f(\omega_k )|}{|ss'\Delta\omega_k|} |\Delta \omega_k|^3\dif s'\dif s>\frac{\sqrt{m\eta\delta}y}4, A
		\Big)\\
		&\le& \PP\big(\sum_{k=0}^{m-1}
		C(1+|\omega_k|^5)|\Delta\omega_k|^31_{\{|\Delta\omega_k|<y_1\}}\ge \sqrt{m\eta\delta}y\big)\\
		&\le& \PP\big(\sum_{k=0}^{m-1}
		C(1+|\omega_k|^5)|\Delta\omega_k|^31_{\{|\Delta\omega_k|<y_1\}}\ge \sqrt{m\eta\delta}y, |\omega_k|<y_2 \text{ for any } k\big)\\
		&&~~+\PP(\max_{k\in\{0,...,m-1\}}|\omega_k|\ge y_2)\\
		&\le& \exp\big\{-\frac{C(\sqrt{m\eta\delta}y-m(\eta\delta)^{\frac32})^2}{my_2^{10}y_1^6}\big\}+Cme^{-y_2^2},
	\end{eqnarray*}
	where the last inequality follows  \cite[Theorem 2]{dedecker2015subgaussian} and the fact $\E|\Delta\omega_k|^3\le C(\eta\delta)^{\frac32}$.

	 For the second term, a straight calculation implies
	\begin{eqnarray*}
		\PP(A^C)&\le&\sum_{k=0}^{m-1}\PP(|\Delta \omega_k|>y_1)\\
		&\le&\sum_{k=0}^{m-1}\PP(\eta|\nabla \psi(\omega_k,\zeta_{k+1})|)>y_1/2)+\sum_{k=0}^{m-1}\PP(\sqrt{\eta\delta}|\xi_{k+1}|>y_1/2)\\
		&\le& \sum_{k=0}^{m-1}\PP( |\omega_k|>\frac{Cy_1}\eta)+\sum_{k=0}^{m-1}\PP(|\nabla\psi(0,\zeta_{k+1})|)>\frac{Cy_1}\eta)+\sum_{k=0}^{m-1}\PP(|\xi_{k+1}|>\frac{Cy_1}{\sqrt{\eta\delta}})\\
		&\le& 2me^{-Cy_1^2/\eta^2}+me^{-Cy_1^2/(\eta\delta)},
	\end{eqnarray*}
	where the second inequality follows the iteration of $\omega_k$, the last inequality follows Lemma \ref{lem:expmoment} and Assumption \ref{assu2}. Combining the calculation above, we can get
	\begin{eqnarray*}
		\PP(\mathcal{R}_{\eta,3}>y/4)\le \exp\big\{-\frac{C(\sqrt{m\eta\delta}y-m(\eta\delta)^{\frac32})^2}{my_2^{10}y_1^6}\big\}+Cme^{-y_2^2}+2me^{-Cy_1^2/\eta^2}+me^{-Cy_1^2/(\eta\delta)}.
	\end{eqnarray*}
	Taking $y_1=y^{1/9}\eta^{7/18}\delta^{7/18} $ and $y_2= y^{1/9}\eta^{-1/9}\delta^{-1/9}$, we complete the proof of \eqref{e:R-3}, that is,
	\begin{eqnarray*}
		\PP\left(\left\vert\mathcal{R}_{\eta,3}\right\vert >y/4\right)& \le & Ce^{-cy^{\frac29}\eta^{-\frac2{9}}\delta^{-\frac2{9}}} ,
	\end{eqnarray*}
	for  $c(\sqrt{m}\eta\delta \vee \eta\delta) <y<C\eta^{-\frac72}\delta^{-\frac72}$.
	
	\textit{d) Control of $\mathcal{R}_{\eta,4}$.} Following the notations of $\Sigma(\omega_k)$ and $\Delta\omega_k$, we have
	\begin{align}\label{e:eqR2}
		&\PP(\mathcal{R}_{\eta,4}>y/4)\\
		=& \PP\big(\frac{1}{2\sqrt{m\eta\delta}}\sum_{k=0}^{m-1}\Ll \nabla^2f(\omega_k ), 
		\eta^2I_{1,k}+\eta\delta I_{2,k}+\eta^{\frac32}\delta^{\frac12} I_{3,k}+\eta^{2} I_{4,k}
		\Rr_\mathrm{HS}>y/4
		\big)\nonumber\\
		\le& \PP\big(\sum_{k=0}^{m-1}\Ll \nabla^2f(\omega_k ), 
		I_{1,k}\Rr_\mathrm{HS}>C m^{\frac12}\delta^{\frac12}\eta^{-\frac32} y
		\big)
		+
		\PP\big(\sum_{k=0}^{m-1}\Ll \nabla^2f(\omega_k ), 
		I_{2,k}\Rr_\mathrm{HS}>C m^{\frac12}\eta^{-\frac12}\delta^{-\frac12} y
		\big)\nonumber\\
		&+
		\PP\big(\sum_{k=0}^{m-1}\Ll \nabla^2f(\omega_k ), 
		I_{3,k}\Rr_\mathrm{HS}>C m^{\frac12}\eta^{-1} y
		\big)+\PP\big(\sum_{k=0}^{m-1}\Ll \nabla^2f(\omega_k ), 
		I_{4,k}\Rr_\mathrm{HS}>C m^{\frac12}\delta^{\frac12}\eta^{-\frac32} y
		\big),\nonumber
	\end{align}
	where
	\begin{eqnarray*}
		I_{1,k}&=&\E_k[\nabla\psi(\omega_k,\zeta_{k+1})\nabla\psi(\omega_k,\zeta_{k+1})^\top]-\nabla\psi(\omega_k,\zeta_{k+1})\nabla\psi(\omega_k,\zeta_{k+1})^\top,
		\\
		I_{2,k}&=& I_d-\xi_{k+1}\xi_{k+1}^\top,\\
		I_{3,k}&=&\nabla \psi(\omega_k,\zeta_{k+1})\xi_{k+1}^\top
		+\xi_{k+1}\nabla \psi(\omega_k,\zeta_{k+1})^\top,\\
		I_{4,k}&=&-\nabla P(\omega_k)\nabla P(\omega_k)^\top.
	\end{eqnarray*}
	
	For the first term of \eqref{e:eqR2}.
	According to \eqref{e:psibound1} and Assumption \ref{assu2},	
	it is easy to verify that
	\begin{eqnarray*}
		\E_{k}|I_{1,k}|^i
		&\le& C^i(1+|\omega_k|^{2i}+i!),
	\end{eqnarray*}
	which satisfies the condition of Lemma \ref{lem:RC} with $\alpha=2$. Thus, \eqref{e:RC2} yields
	\begin{eqnarray}\label{e:eqR3}
		\PP\big(\sum_{k=0}^{m-1}\Ll \nabla^2f(\omega_k ), 
		I_{1,k}\Rr_\mathrm{HS}>C m^{\frac12}\delta^{\frac12}\eta^{-\frac32} y
		\big)	\le C\exp\{-c\delta^\frac17 y^\frac27\eta^{-\frac37}\}
	\end{eqnarray}
as $(\eta/\delta)^{1/2}<y$. Similar to the estimation of \eqref{e:eqR3}, one can also verify that $I_{2,k}$ and $I_{3,k}$ satisfy the condition \eqref{e:RC} with $\alpha=0$ and $\alpha=1$ respectively, thus \eqref{e:RC2} implies
	\begin{eqnarray}\label{e:eqR4}
		\PP\big(\sum_{k=0}^{m-1}\Ll \nabla^2f(\omega_k ), I_{2,k}\Rr_\mathrm{HS}>C m^{\frac12}\eta^{-\frac12}\delta^{-\frac12} y
		\big)\le C\exp\{-c\eta^{-\frac15}\delta^{-\frac15}y^{\frac25}\},
	\end{eqnarray}
	and
	\begin{eqnarray}\label{e:eqR5}
		\PP\big(\sum_{k=0}^{m-1}\Ll \nabla^2f(\omega_k ), I_{3,k}\Rr_\mathrm{HS}>C m^{\frac12}\eta^{-1} y\big)\le C\exp\{-cy^{\frac13}\eta^{-\frac13}\}.
	\end{eqnarray}
	For the last term of \eqref{e:eqR2}, let $\hat\omega_k=\omega_k1_{\{|\omega_k|\le y_3\}}$. Similar to the estimation of \eqref{e:RC1}, \eqref{e:pbound1} and \eqref{e:regularity3} yield
	\begin{eqnarray*}
		&&\PP\big(\sum_{k=0}^{m-1}\Ll \nabla^2f(\omega_k ), 
		I_{4,k}\Rr_\mathrm{HS}>C m^{\frac12}\delta^{\frac12}\eta^{-\frac32} y
		\big)\\
		&\le&  \PP\big(\sum_{k=0}^{m-1}\Ll \nabla^2f(\hat\omega_k ), 
		-\nabla P(\hat\omega_k)\nabla P(\hat\omega_k)^\top\Rr_\mathrm{HS}>C m^{\frac12}\delta^{\frac12}\eta^{-\frac32} y\big)+\sum_{k=0}^{m}\PP(|\omega_k|\ge y_3)\\
		&\le&  \PP\big(\sum_{k=0}^{m-1}
		(1+|\hat\omega_k|^6)>C m^{\frac12}\delta^{\frac12}\eta^{-\frac32} y\big)+mCe^{-cy_3^2}.
	\end{eqnarray*}
	For the above probability, we have
	\begin{eqnarray*}
		&&  \PP\big(\sum_{k=0}^{m-1}
		(1+|\hat\omega_k|^6)>C m^{\frac12}\delta^{\frac12}\eta^{-\frac32} y\big)\\
		&=&\PP\big(\sum_{k=0}^{m-1}
		(|\hat\omega_k|^6-\E|\hat\omega_k|^6)>Cm^{\frac12}\delta^{\frac12}\eta^{-\frac32} y-m-\sum_{k=0}^{m-1}\E|\hat\omega_k|^6\big)\\
		&\le&  \PP\big(\sum_{k=0}^{m-1}
		(|\hat\omega_k|^6-\E|\hat\omega_k|^6)>C m^{\frac12}\delta^{\frac12}\eta^{-\frac32} y-m)\big)\\
		&\le& \exp\{-y^2\delta\eta^{-3}y_3^{-12}\}.
	\end{eqnarray*}
	Thus,
	\begin{eqnarray}\label{e:eqR7}
		\PP\big(\sum_{k=0}^{m-1}\Ll \nabla^2f(\omega_k ), 
		I_{4,k}\Rr_\mathrm{HS}>C m^{\frac12}\delta^{\frac12}\eta^{-\frac32} y
		\big)\le C\exp\{-cy^{\frac27}\delta^{\frac17}\eta^{-\frac37}\},
	\end{eqnarray}
	by taking $y_3=(y^2\eta^{-3}\delta)^{1/14}$ and $y>m^{\frac12}\eta^{\frac32}\delta^{-\frac12}$. Combing the result of \eqref{e:eqR4}-\eqref{e:eqR7}, we can get the bound of \eqref{e:eqR2}, that is
	\begin{eqnarray*}
		&&\PP(|\mathcal{R}_{\eta,4}|>y/4)\le	Ce^{-cy^{\frac27}\delta^{\frac17}\eta^{-\frac37}}+Ce^{-cy^{\frac25}\delta^{-\frac15}\eta^{-\frac15}}.
	\end{eqnarray*}

Thus we obtain that
	\begin{equation*}
	\PP(|\mathcal{R}_\eta|>y)\le C\Big( e^{-cy\eta^{\frac12}\delta^\frac12m^{\frac12} }
	+e^{-cy^\frac25\delta^\frac15 \eta^{-\frac1{5}}}
	+e^{-cy^{\frac29}\eta^{-\frac29}\delta^{-\frac29} }
	+e^{-cy^{\frac27}\delta^{\frac17}\eta^{-\frac37}}\Big)
\end{equation*}
for $c(\eta^{\frac12}\delta^{-\frac12}\vee m^{\frac12}\eta\delta)\le y\le C\eta^{-\frac72}\delta^{-\frac72}$.

\end{proof}

\begin{appendix}

\section{Proof of Lemma \ref{lem:1}}\label{appen:lem1}
\begin{proof}[Proof of Lemma \ref{lem:1}]
	Since $\nabla P(x)=\E[\nabla\psi(x,\zeta)]$, it is easy to see that $\nabla P$ has the same properties, that is,
	\begin{eqnarray*}
		|\nabla P(x)-\nabla P(y)|\le L|x-y|,
	\end{eqnarray*}
	\begin{eqnarray*}
		\Ll x-y,-\nabla P(x)+\nabla P(y)\Rr\le -K_1|x-y|^2+K_2,
	\end{eqnarray*}
	for any $x,y\in \R^d$.
	
	Following the assumptions \eqref{e:lip} and \eqref{e:dissi}, we can further have the bounds for $\nabla \psi(x,z)$ and $\nabla P(x)$, that is,
	\begin{eqnarray*}
		|\nabla\psi(x,\zeta)|\le L|x|+|\nabla\psi(0,\zeta)|,
	\end{eqnarray*} 
	\begin{eqnarray*}
		|\nabla P(x)|\le L|x|+|\nabla P(0)|.
	\end{eqnarray*}
	Assumptions \eqref{e:dissi}, \eqref{e:pdissi} and Young's inequality imply
	\begin{align}\label{e:psibound2}
		\Ll x, -\nabla \psi(x,\zeta)\Rr&=\Ll x-0, -\nabla \psi(x,\zeta)+\nabla \psi(0,\zeta)\Rr- \Ll x, \nabla \psi(0,\zeta)\Rr   \nonumber \\
		&\le-K_1|x|^2+K_2+\frac{K_1}{2}|x|^2+\frac{1}{2K_1}|\nabla\psi(0,\zeta)|^2 \nonumber\\
		&=-\frac{K_1}{2}|x|^2+K_2+\frac{1}{2K_1}|\nabla\psi(0,\zeta)|^2.
	\end{align}
	Similarly,	
	\begin{align}\label{e:pbound2}
		\Ll x, -\nabla P(x)\Rr
		&\le -\frac{K_1}{2}|x|^2+K_2+\frac{1}{2K_1}|\nabla P(0)|^2
	\end{align}
	Moreover,
	\begin{eqnarray}\label{e:sbound}
		\|\Sigma(x)\|\le 2\E|\nabla\psi(x,\zeta)|^2\le 4L^2|x|^2+C.
	\end{eqnarray}
	
	For the Lipschitz property of $Q_{\eta,\delta}$, recall that	
	\begin{equation*}
		Q_{\eta,\delta}(x)=\big(\E[V_{\eta,\delta}(x,\zeta,\xi)V_{\eta,\delta}(x,\zeta,\xi)^\top]\big)^{\frac12}
	\end{equation*}
	By Assumption \ref{e:lip} and \ref{e:plip}, the definition of $V_{\eta,\delta}(x,\zeta,\xi)$ implies that
	\begin{eqnarray*}
		|V_{\eta,\delta}(x,\zeta,\xi)-V_{\eta,\delta}(y,\zeta,\xi)|\le 2\sqrt\eta L|x-y|,	
	\end{eqnarray*}
	which is Lipschitz. Denote the L2 norm $\|X\|_{L^2}=(\E\|X\|^2)^{\frac12}$ for any random variable $X$. Then we have
	\begin{eqnarray*}
		Q_{\eta,\delta}(x)=\| V_{\eta,\delta}(x,\zeta,\xi)V_{\eta,\delta}(x,\zeta,\xi)^\top)^{\frac12}  \|_{L^2}
	\end{eqnarray*}
	Thus,
	\begin{eqnarray*}
		&&\|Q_{\eta,\delta}(x)-Q_{\eta,\delta}(y)\|\\
		&=&\big\| \| V_{\eta,\delta}(x,\zeta,\xi)V_{\eta,\delta}(x,\zeta,\xi)^\top)^{\frac12}  \|_{L^2}-\| V_{\eta,\delta}(y,\zeta,\xi)V_{\eta,\delta}(y,\zeta,\xi)^\top)^{\frac12}  \|_{L^2}   \big\|\\
		&\le&\| ( V_{\eta,\delta}(x,\zeta,\xi)V_{\eta,\delta}(x,\zeta,\xi)^\top)^{\frac12}  - (V_{\eta,\delta}(y,\zeta,\xi)V_{\eta,\delta}(y,\zeta,\xi)^\top)^{\frac12}  \|_{L^2}   
	\end{eqnarray*}
	Since the mapping $V_{\eta,\delta}\to \big(V_{\eta,\delta}V_{\eta,\delta}^\top \big)^{\frac12}=V_{\eta,\delta}V_{\eta,\delta}^\top/|V_{\eta,\delta}|$ is Lipschitz, we have
	\begin{eqnarray*}
		\|Q_{\eta,\delta}(x)-Q_{\eta,\delta}(y)\|\le C\sqrt \eta |x-y|.
	\end{eqnarray*}
\end{proof}

\section{Proof of Ergodicity}\label{appen:ergodic}
\begin{proof}[Poor of Lemma  \ref{lem:ergodic}]
	We first give the proof of the ergodicity of $(X_t)_{t\ge0}$.  For the Lyapunov function $V(x)=|x|^2+1$ on $\R^d$, \eqref{e:generator}, \eqref{e:pbound2} and \eqref{e:sbound} imply
	\begin{eqnarray*}
		\mathcal L V(x)&=&-\Ll \nabla P(x), 2x\Rr+\Ll \eta\Sigma(x)+\delta I_d, I_d  \Rr_{\textrm {HS}}\\
		&\le& -K_1|x|^2+4\eta L^2|x|^2+C.
	\end{eqnarray*}
	For small enough $\eta\le K_1/(8L^2)$, one has
	\begin{eqnarray*}
		\mathcal L V(x) &\le& -\frac {K_1}4V(x) +(C+\frac {K_1}4)1_{\{|x|^2\le K_1+4C\}}.
	\end{eqnarray*}
	By \cite[Theorem 6.1]{MeTw93}, $(X_t)_{t\ge0}$  is exponential ergodic with invariant measure $\pi$, that is, there esxist constant $C$ and $c$ such that
	\begin{eqnarray}\label{e:ergod1}
		\sup_{|h|\le V}|\E h(X_t^x)-\pi(h)|\le CV(x)e^{-ct}.
	\end{eqnarray}
	
	The ergodicity of $(\omega_k)_{k\ge0}$ follows  \cite[Theorem 2.1]{Tweedie94}. Notice that
	\begin{eqnarray*}
		\E_k[V(\omega_{k+1})]&=& 1+ \E_{k}|\omega_k-\eta\nabla\psi(\omega_k,\zeta_{k+1})+\sqrt{\eta\delta}\xi_{k+1}|^2\\
		&=& 1+ |\omega_k|^2+\eta^2\E_{k}|\nabla\psi(\omega_k,\zeta_{k+1})|^2+\eta\delta d- 2\eta\Ll \omega_k, \nabla P(\omega_k)\Rr\\
		&\le& (1+2\eta^2L^2-\eta K_1)|\omega_k|^2+1+C\eta.
	\end{eqnarray*}
	Denote the transition probability of $(\omega_k)_{k\ge0}$ by $P(x,\dif y)$ for $x,y\in\R^d$ and let
	\begin{align*}
		\quad V^n (x)=e^{c_1 n\eta}V(x),\quad r(n)= c_1\eta e^{c_1 n\eta}.
	\end{align*}
	A straightforward calculation implies 
	\begin{eqnarray*}
		&&PV^{n+1}(x)+r(n)V(x)\\
		&=&e^{c_1(n+1)\eta}P V(x)+c_1\eta e^{c_1n\eta}V(x)\\
		&\le& e^{c_1(n+1)\eta}\big((1+2\eta^2L^2-\eta K_1)|x|^2+1+C\eta\big)+c_1\eta e^{c_1n\eta}V(x)\\
		&=& e^{c_1n\eta}V(x)+c_1\eta e^{c_1n\eta}\big[\big(\frac {e^{c_1\eta}}{c_1\eta}(1+2\eta^2L^2-\eta K_1)+1-\frac{1}{c_1\eta}\big)V(x)+\frac {e^{c_1\eta}}{c_1\eta}(C\eta+\eta K_1-2\eta^2L^2)\big]\\
		&=& V^n(x)+r(n)\big[\frac {1}{c_1\eta}\big(e^{c_1\eta}(1+2\eta^2L^2-\eta K_1)+c_1\eta-1\big)V(x)+\frac {e^{c_1\eta}}{c_1\eta}(C\eta+\eta K_1-2\eta^2L^2)\big].
	\end{eqnarray*}
	Choosing $\eta$ small enough such that $e^{c_1\eta}(1+2\eta^2L^2-\eta K_1)+c_1\eta<1$, we can get
	\begin{eqnarray*}
		P V^{n+1}(x)+r(n)V(x)
		&\le& V^n(x)+br(n)1_{\{x\in \mathcal{C}\}},
	\end{eqnarray*}
	where $b=\frac {e^{c_1\eta}}{c_1\eta}(C\eta+\eta K_1-2\eta^2L^2)$ and $\mathcal{C}=\{x:V(x)\le \frac{  e^{c_1\eta}(C\eta+\eta K_1-2\eta^2 L^2)}{1-e^{c_1\eta}(1+2\eta^2L^2-2\eta K_1)-c_1\eta}\}$.  \cite[Theorem 2.1]{Tweedie94} implies that $(\omega_k)_{k\ge0}$ is ergodic with invariant measure $\pi_\eta$, that is, there exist constant $C$ and $c$ such that
	\begin{eqnarray}\label{e:ergod2}
		\sup_{|h|\le V}|\E h(\omega_k^x)-\pi_\eta(h)|\le C\eta^{-1}V(x)e^{-ck\eta}.
	\end{eqnarray} 	
\end{proof}
\end{appendix}

\section*{Acknowledgements}
Hongsheng Dai is supported by the EPSRC research grant "Pooling INference and COmbining Distributions Exactly: A Bayesian approach (PINCODE)", EP/X027872/1, and by the UKRI grant EP/Y014650/1 as part of the ERC Synergy project OCEAN. Xiequan Fan is partially supported by the National Natural Science Foundation of China (Grant No. 12371155). For the purpose of open access, the authors have applied a Creative Commons Attribution (CC BY) license to any Author Accepted Manuscript version arising from this submission.

\bibliographystyle{plainnat}
\bibliography{reference}
\end{document}